\documentclass[onefignum,onetabnum]{siamonline171218}
\usepackage{hyperref}
\usepackage{amsfonts,amssymb}
\usepackage{enumitem}
\usepackage{pifont}
\usepackage{booktabs}       
\usepackage{multirow}
\usepackage{mathtools}
\usepackage{makecell}
\usepackage{algorithm}
\usepackage{algpseudocode}
\usepackage{amsmath}
\usepackage{xcolor}
\usepackage{tabu}
\usepackage{hyperref}
\makeatletter
\def\etal{\textit{et al. }}
\newtheorem{remark}{Remark}

\newcommand{\bosy}{\boldsymbol}
\renewcommand{\b}{\textbf}
\title{Coupled Reconstruction of 2D Blood Flow and Vessel Geometry from Noisy Images via Physics-Informed Neural Networks and Quasi-Conformal Mapping\thanks{The work of Xue-Cheng Tai is partially supported by NORCE Kompetanseoppbygging program. The work of Jean-Michel Morel is partially supported by RGC-GRF project 11309925, Mathematical Formalization of GIS. The work of Raymond H. Chan is partially supported by HKRGC Grants CityU11309922, LU13300125, ITF Grant No. MHP/054/22, and LU BGR105824. Thanks to the InnoHK project at Hong Kong Centre for Cerebro-cardiovascular 
Health Engineering (COCHE).}}

\author{
Han Zhang \thanks{Department of Mathematics, City University of Hong Kong and Hong Kong Centre for Cerebro-Cardiovascular Health Engineering, Hong Kong, China.
(\email{hzhang863-c@my.cityu.edu.hk})}
\and Xue-Cheng Tai \thanks{Norwegian Research Centre (NORCE), Nygardsgaten 112, 5008 Bergen, Norway. 
(\email{xtai@norceresearch.no})}
\and Jean-Michel Morel \thanks{Department of Mathematics, City University of Hong Kong, Hong Kong, China.
(\email{jeamorel@cityu.edu.hk})}
\and Raymond H. Chan \thanks{Department of Operations and Risk Management and the School of Data Science, Lingnan University and Hong Kong Centre for Cerebro-Cardiovascular Health Engineering, Hong Kong, China. 
(\email{raymond.chan@ln.edu.hk})}
}

\begin{document}

\maketitle
\begin{abstract}
Blood flow imaging provides important information for hemodynamic behavior within the vascular system and plays an essential role in medical diagnosis and treatment planning. However, obtaining high-quality flow images remains a significant challenge. In this work, we address the problem of denoising flow images that may suffer from artifacts due to short acquisition times or device-induced errors. We formulate this task as an optimization problem, where the objective is to minimize the discrepancy between the modeled velocity field, constrained to satisfy the Navier-Stokes equations, and the observed noisy velocity data. To solve this problem, we decompose it into two subproblems: a fluid subproblem and a geometry subproblem. The fluid subproblem leverages a Physics-Informed Neural Network to reconstruct the velocity field from noisy observations, assuming a fixed domain. The geometry subproblem aims to infer the underlying flow region by optimizing a quasi-conformal mapping that deforms a reference domain. These two subproblems are solved in an alternating Gauss-Seidel fashion, iteratively refining both the velocity field and the domain. Upon convergence, the framework yields a high-quality reconstruction of the flow image. We validate the proposed method through experiments on synthetic flow data in a converging channel geometry under varying levels of Gaussian noise, and on real-like flow data in an aortic geometry with signal-dependent noise. The results demonstrate the effectiveness and robustness of the approach. Additionally, ablation studies are conducted to assess the influence of key hyperparameters.
\end{abstract}
\begin{keyword}
  Flow Image Reconstruction, Blood Flow Simulation, Physics-Informed Neural Network, Quasi-Conformal Geometry, Image Denoise, Image Segmentation
\end{keyword}

\begin{AMS}
  65N21, 65D18, 76D05, 92C50
\end{AMS}

\section{Introduction}

Blood flow imaging plays a fundamental role in characterizing cardiovascular dynamics, offering critical insights into hemodynamic behavior within the vascular system. This technique is indispensable for medical monitoring, cardiovascular disease diagnosis, and surgical planning. As a non-invasive approach, it enables safe assessment of patients with severe cardiovascular conditions while avoiding the flow disturbances caused by intravascular instrumentation \cite{yan2022impact}.

However, achieving accurate and high-resolution blood flow imaging remains a significant challenge. Existing modalities, such as flow MRI, often require long acquisition times \cite{markl20124d} to produce clear and high-quality flow images. This not only causes discomfort for patients but also makes the imaging process highly susceptible to motion artifacts, such as those caused by respiration. On the other hand, fast acquisition approaches such as low-dose flow MRI or Doppler ultrasound offer faster imaging but typically suffer from high noise levels and signal loss, resulting in data that may be unreliable for precise diagnostic purposes.

In this work, we focus on the problem of reconstructing flow images that are often corrupted by artifacts resulting from short acquisition times or device-induced noise. In this work, we follow the pioneering word done at  \cite{kontogiannis2022joint,kontogiannis2024bayesian}. They have modelled this problem by formulating it as an optimization problem which minimizes the discrepancy between velocity fields, constrained by the incompressible Navier-Stokes equations, and the noisy observations.
 
To more effectively tackle this problem, we propose a reformulation that decomposes into two coupled subproblems: a fluid subproblem and a geometry subproblem. In the fluid subproblem, the flow dynamics are reconstructed within a fixed domain by minimizing the discrepancy between the modeled and measured velocity fields, while ensuring physical consistency through the incompressible Navier-Stokes constraints. In the geometry subproblem, a correction mapping, interpretable as a domain gradient, is estimated to register the reconstructed velocity to the observed data, thereby capturing domain misalignments and approximating the gradient with respect to the spatial domain. By alternating between the two subproblems in a Gauss-Seidel fashion, the overall optimization proceeds iteratively until convergence, which then solves the original problem with theoretical support.

To efficiently solve this coupled optimization problem while leveraging the advantages of differentiable architectures, we design a new algorithm by making the entire framework to be network-based. For the fluid subproblem, a novel Physics-Informed Neural Network (PINN) is used to reconstruct the velocity field, allowing the Navier-Stokes constraints to be enforced directly through backpropagation. For the geometry subproblem, a U-Net architecture is adopted to predict the quasi-conformal mapping, which iteratively corrects geometric misalignments by minimizing the discrepancy between the measured and reconstructed velocity fields through a correction mapping.

We validate our method through comprehensive experiments. First, we evaluate performance on synthetic flow images in a converging channel geometry under both low and high Gaussian noise levels. Next, we test the approach on more anatomically realistic aorta geometries using a signal-dependent noise model. In all cases, our method produces more accurate and higher-quality reconstructions compared to conventional baselines. We also performed ablation studies to assess the impact of different weighting parameters in the optimization of the fluid and geometry subproblems, identifying effective hyperparameter choices. Note that our work is presented and implemented in the 2D setting; however, it can be potentially extended to 3D by replacing the Beltrami coefficient with 3D conformality measures in \cite{zhang2021topology}.

Overall, the contribution of this task could be summarized as follows.
\begin{itemize}
    \item We propose a novel framework for jointly reconstructing both the blood flow field and the flow domain geometry from noisy velocity images, by integrating the Navier-Stokes model with quasi-conformal mapping theory.
    \item We reformulate the inverse Navier-Stokes problem into two decoupled subproblems: a fluid subproblem and a geometry subproblem. Theoretical analysis is provided to justify the validity of this decomposition.
    \item We develop a method that utilizes Physics-Informed Neural Networks to reconstruct the velocity field by fitting to noisy measurement data and enforcing the Navier-Stokes equations as soft physical constraints.
    \item We introduce a mapping approach for segmenting the flow region using quasi-conformal geometry, which enables globally guided domain evolvation and leveraging morphological priors from the template.
\end{itemize}

\section{Related Work}

\subsection{Blood Simulation}
Accurate blood flow measurement is vital for cardiovascular healthcare, informing diagnosis, surgical planning, and treatment. Although invasive methods such as Fractional Flow Reserve (FFR) \cite{pijls1996measurement} are widely used, they pose risks such as complications and can disrupt the dynamics of natural flow \cite{yan2022impact}, leading to measurement inaccuracies. Numerical simulations offer a noninvasive alternative by modeling blood as an incompressible Newtonian fluid governed by the Navier-Stokes equations \cite{peskin1972flow}. To account for vessel deformation, these models are extended to fluid-structure interaction (FSI) frameworks \cite{quarteroni2004mathematical}.

Methods have been proposed to solve FSI problems. These include dimensional coupling \cite{formaggia2001coupling}, parallel preconditioning \cite{crosetto2011parallel}, and splitting schemes with stability guarantees \cite{bukavc2013fluid}. Other approaches involve space-time formulations \cite{bazilevs2008isogeometric}, discontinuous Galerkin methods \cite{wang2018higher}, coupled momentum methods \cite{figueroa2006coupled}, and learning-based approach \cite{zhang2024meshless,zhang2024full}.

\subsection{Image Denoise}
Image restoration is challenged by noise types such as speckle, Poisson, and Rician noise, which degrade quality and complicate analysis. To address these, a variety of denoising methods have been developed. Lou et al. \cite{lou2015weighted} proposed a weighted difference between anisotropic and isotropic total variation (TV) regularization. Li et al. \cite{li2016variational} introduced a variational model for multiplicative noise removal based on the difference of convex functions \cite{TAO1986249} combined with primal-dual algorithms. Xiao et al. \cite{xiao2011restoration} further improved impulsive noise removal by incorporating adaptive proximal parameters into a nonconvex TV-log model. To reduce dependence on initialization, Getreuer et al. \cite{getreuer2011variational} developed the convex Getreuer-Tong-Vese (GTV) model for MAP-Rician noise, efficiently solved using Bregman splitting \cite{goldstein2009split}. Chen et al. \cite{chen2015convex} proposed the strictly convex CZ model, which adds a quadratic regularization term for stability. 

Beyond these convex approaches, more advanced techniques employing higher-order and nonconvex regularization have also been explored. Kang et al. \cite{kang2015nonconvex} proposed a spatially adaptive nonconvex prior to better preserve fine structures, and Martin et al. \cite{martin20171} demonstrated effective Rician noise removal using nonconvex fidelity terms in conjunction with proximal point algorithms. These developments reflect the ongoing effort to improve robustness, structural preservation, and reconstruction accuracy in challenging noise environments.

\subsection{Physics-Based Reconstruction}
The reconstruction of flow images \cite{corona2021joint}, often corrupted by artifacts from short acquisition times or device noise, has been addressed using physics-based approaches. Kontogiannis \etal \cite{kontogiannis2022joint} formulated the inverse Navier-Stokes problem, which optimizes over the fluid dynamics and the computation domain by minimizing the discrepancy between modeled velocity fields and observations under incompressible Navier-Stokes constraints. This was extended to 3D flows in \cite{kontogiannis2024bayesian}, introducing a stabilized Nitsche cut-cell FEM for high-Reynolds-number flows, an implicit geometry representation with a viscous Eikonal constraint.

In parallel, Aguayo \etal \cite{aguayo2022analysis, aguayo2022stability, aguayo2023distributed} developed methods for modeling obstacles in incompressible flows using fictitious domains, penalization potentials, and Brinkman-type permeability terms for distributed-resistance treatment. For problems with known geometries, physics-informed neural networks have been applied to reconstruct flow fields: Sun \etal \cite{sun2020physics} proposed a physics-constrained neural network to reconstruct flow fields from sparse and noisy measurements, combining probabilistic modeling of physical constraints with data uncertainty to guide the network toward physically consistent solutions. Gao \etal \cite{gao2021super} developed a CNN-based framework for super-resolution and denoising of fluid flows, using physics-informed losses to reconstruct high-resolution fields from coarse or noisy observations without high-resolution labels. Wang \etal \cite{wang2022dense} employed physics-informed neural networks regularized by the Navier-Stokes equations for dense flow reconstruction, but their method does not account for boundary conditions or domain estimation during the reconstruction process.

The problem investigated in this study is inspired by the work of \cite{kontogiannis2022joint}, but our approach introduces several key differences. First, we reformulate the original problem into two alternating solved coupled subproblems explicitly. Such a reformulation differs from previous approaches and is supported by a dedicated theoretical result (Theorem \ref{thm:main}), ensuring that it is mathematically well founded. Second, we solve the fluid subproblem in a new way by using a physics-informed neural network (PINN), which removes the need for explicit numerical discretization, making the approach meshless and allowing the integration of modern deep learning techniques. Third, the geometry subproblem is newly formulated as a registration task, where the domain evolution is represented by a learned deformation map. Fourth, we incorporate quasi-conformal (QC) theory into the geometric component to ensure topology preservation and naturally embed morphological priors in the reconstruction process.

\section{Inverse Navier-Stokes problem}
\label{sec:problem}

\begin{table}[!htp]
\centering
\begin{tabular}{c|l||c|l}
\hline
\textbf{Notation}   & \textbf{Description}      & \textbf{Notation}     & \textbf{Description} \\\hline
$\bosy{x}$          & Spatial point             & $c$                   & Correction mapping\\
$\bosy{u}$          & Modeled flow velocity     & $f$                   & Deformation mapping\\
$p$                 & Kinetic pressure          & $\Omega_0$            & Reference flow domain \\
$\nu$               & Kinematic viscosity       & $\Omega$              & Flow domain            \\
$\bosy{n}$          & Outward normal vector     & $\Gamma^\text{i}$     & Inlet boundary \\
$\bosy{g}$          & Inlet boundary condition  & $\Gamma^\text{o}$     & Outlet boundary        \\
$\tilde{\bosy{u}}$  & Measured flow velocity    & $\Gamma^\text{w}$     & Wall boundary \\
\hline\hline
\end{tabular}
\caption{Notation table for problem formulation.}
\label{tb:notation}
\end{table}

The main objective of this work is to reconstruct the velocity field over the vessel lumen region, given only a measured blood flow image. To do so, the model should not only try to extract information from the measured data but also try to satisfy its fluid physics, which is generally described by the incompressible Navier-Stokes equation \cite{quarteroni2004mathematical}. Before presenting the model, we first describe the notation that is needed for the discussion. Let $D$ denote the full image domain, and let $\tilde{\boldsymbol{u}}$ represent the noisy image obtained from measurements. We describe the lumen domain of the blood vessel as $\Omega\subset D$. Within this domain $\Omega$. The boundary of the vessel domain $\partial \Omega$ is divided into three parts: the inlet $\Gamma^\text{i}$, the outlet $\Gamma^\text{o}$, and the wall of the vessel  $\Gamma^\text{w}$, satisfying $\partial \Omega = \Gamma^\text{i} \cup \Gamma^\text{o} \cup \Gamma^\text{w}$.

In this inverse problem setting, only the noisy velocity field $\tilde{\boldsymbol{u}}$ is available from the imaging data, while the underlying flow domain $\Omega$, the inlet boundary condition $\boldsymbol{g} \in C(\Gamma^\text{i})$, and the true velocity field $\boldsymbol{u}$ are unknown. The goal is to reconstruct $\Omega$, $\boldsymbol{g}$, and $\boldsymbol{u}$ such that the modeled velocity field is consistent with the observed data. To quantify the data misfit, we define a discrepancy functional that measures the $L^2$ distance between the measured velocity and the simulated one as
\begin{equation}
\mathbf{L}_\text{data}(\boldsymbol{u}) = \|\tilde{\boldsymbol{u}} - \boldsymbol{u}\|_{L^2(D)}^2,
\label{eq:data_error}
\end{equation}
The modeled velocity field is required to satisfy the steady incompressible Navier–Stokes equations in the domain $\Omega$. Specifically, the velocity $\boldsymbol{u}$ associated with a  pressure $p$ satisfies the system
\begin{equation}
\begin{aligned}
\left\{\begin{aligned}
    \boldsymbol{u} \cdot \nabla \boldsymbol{u} -  \nabla p + \nu \Delta \boldsymbol{u} =\boldsymbol{0}\quad &\text{ in } \Omega, \\
    \nabla\cdot \boldsymbol{u} = 0\quad &\text{ in } \Omega,\\
    \boldsymbol{u} = \boldsymbol{g} \quad &\text{ on } \Gamma^\text{i},\\
    (- p + \nu \nabla \boldsymbol{u})\cdot \boldsymbol{n} = \boldsymbol{0} \quad &\text{ on } \Gamma^\text{o},\\
    \boldsymbol{u} = \boldsymbol{0} \quad &\text{ on } \Gamma^\text{w}.\\
\end{aligned}\right.
\end{aligned}
\label{eq:ns}
\end{equation}
for a domain $\Omega \subset D$, where $\nu$ denotes the kinematic viscosity and $\boldsymbol{n}$ is the outward normal vector on the boundary. With this forward model \eqref{eq:ns}, we could define a solution operator $\mathcal{U}_\Omega$ that maps the inlet condition $\boldsymbol{g}$ and the flow domain $\Omega$ to the corresponding velocity field $\boldsymbol{u}$, that is, $\mathcal{U}_\Omega: (\boldsymbol{g}, \Omega) \mapsto \boldsymbol{u}$ \cite{quarteroni2004mathematical}. 

With these settings, we pose the inverse reconstruction problem as a PDE-constrained optimization problem. Specifically, our objective is to recover the domain $\Omega$ and the inlet boundary condition $\boldsymbol{g}$ by minimizing the discrepancy functional subject to the Navier–Stokes equations, that is, we solve 
\begin{equation}
\begin{aligned}
& \min\limits_{\boldsymbol{g},\Omega} \quad \mathbf{L}_\text{data}(\boldsymbol{u}=\mathcal{U}_\Omega\left[\boldsymbol{g},\Omega \right])
\end{aligned}
\end{equation}
However, optimizing the flow domain $\Omega$ is not straightforward and a specific representation is required. Previous works have employed level-set methods \cite{kontogiannis2022joint} and phase-field approaches \cite{aguayo2021distributed}. In our approach, we model the domain using a quasi-conformal mapping, allowing for both topological control of the evolving process and seamless integration into a network-based optimization framework.

For the quasi-conformal mapping that will be introduced later, we need to use a reference domain $\Omega_0$ that has the same topology as the target domain $\Omega$. The actual flow domain $\Omega$ is then expressed by a quasi-conformal mapping $f$ as
\begin{equation}
    \Omega = f(\Omega_0). 
    \label{eq:mapping_domain}
\end{equation}
As a result, with $\partial \Omega_0 = \Gamma_0^\text{i} \cup \Gamma_0^\text{o} \cup \Gamma_0^\text{w}$, the boundaries of the deformed domain are expressed in terms of the reference boundaries as
\begin{equation}
\Gamma^\text{i} = f(\Gamma_0^\text{i}), \quad 
\Gamma^\text{o} = f(\Gamma_0^\text{o}), \quad 
\Gamma^\text{w} = f(\Gamma_0^\text{w}),
\label{eq:mapping_boundary}
\end{equation}
Consequently, we redefine the solution operator to account for the deformation, denoted by $\mathcal{U}_f : (\boldsymbol{g},f) \mapsto \boldsymbol{u}$, mapping a boundary condition $\boldsymbol{g}$ and a quasi-conformal mapping function $f$  to a velocity field $\boldsymbol{u}$ satisfying (\ref{eq:ns}), (\ref{eq:mapping_domain}) and (\ref{eq:mapping_boundary}). The data-fidelity-driven optimization problem is then reformulated as
\begin{equation}
\min\limits_{\boldsymbol{g},f} \quad \mathbf{L}_\text{data}(\boldsymbol{u}=\mathcal{U}_f\left[\boldsymbol{g},f \right]),
\label{eq:opt_reform}
\end{equation}
subject to the Navier–Stokes equations
\begin{equation}
\begin{aligned}
\left\{\begin{aligned}
    \boldsymbol{u} \cdot \nabla \boldsymbol{u} -  \nabla p + \nu \Delta \boldsymbol{u} =\boldsymbol{0}\quad &\text{ in } f(\Omega_0), \\
    \nabla\cdot \boldsymbol{u} = 0\quad &\text{ in } f(\Omega_0),\\
    \boldsymbol{u} = \boldsymbol{g} \quad &\text{ on } f(\Gamma^\text{i}_0),\\
    (- p + \nu \nabla \boldsymbol{u})\cdot \boldsymbol{n} = \boldsymbol{0} \quad &\text{ on } f(\Gamma^\text{o}_0),\\
    \boldsymbol{u} = \boldsymbol{0} \quad &\text{ on } f(\Gamma^\text{w}_0).\\
\end{aligned}\right.
\end{aligned}
\label{eq:ns_reform}
\end{equation}

Despite this reformulation, the joint optimization over multiple coupled variables remains challenging, especially given the dependency on a deforming spatial domain. To mitigate this complexity, we adopt a Gauss–Seidel-style alternating optimization strategy. The optimization is split into two subproblems. The first, termed the fluid subproblem, fixes the current domain and solves for the velocity field $\boldsymbol{u}$, pressure $p$, and the inlet condition $\boldsymbol{g}$. The second, referred to as the geometry subproblem, holds $\boldsymbol{u}$ and $\boldsymbol{g}$ fixed while optimizing the deformation mapping $f$, thus updating the flow domain.

\subsection{Fluid Problem}
\label{sec:problem_fluid}

For the fluid subproblem, we assume that the domain $\Omega$ is fixed by a known mapping $f$, i.e., $\Omega = f(\Omega_0)$. With the geometry held constant, it becomes more tractable to minimize the discrepancy between the observed velocity field $\tilde{\boldsymbol{u}}$ and the modeled solution $\boldsymbol{u}$, which satisfies the Navier–Stokes equations on the fixed domain. In this setting, we redefine the solution operator $\mathcal{U}_g \colon \boldsymbol{g} \mapsto \boldsymbol{u}$ that maps a given inlet boundary condition $\boldsymbol{g}$ to the corresponding velocity field $\boldsymbol{u}$, while implicitly incorporating the fixed mapping $f$. The resulting optimization problem becomes a partial minimization over $\boldsymbol{g}$ with  $\boldsymbol{u} = \mathcal{U}_g\left[\boldsymbol{g}\right]$, formulated as
\begin{equation}
\begin{aligned}
& \min\limits_{\boldsymbol{g}} \quad  \mathbf{L}_\text{data}(\boldsymbol{u}=\mathcal{U}_g\left[\boldsymbol{g}\right]).  
\end{aligned}
\label{eq:fluid}
\end{equation}
Here, the deformation map $f$ is fixed for this minimization problem, reducing the minimization to a problem that only solves for the fluid dynamics.

\subsection{Geometry Problem}

To model the geometry subproblem, which aims to improve the accuracy of domain estimation, it is essential to derive information that guides the optimization of the deformation mapping. Note that although the observed velocity field is corrupted by noise, it still reflects the flow in the true physical domain. Specifically, the reconstructed flow field that satisfies \eqref{eq:ns_reform} represents the template domain, as it is defined over the region $f(\Omega_0)$. Aligning this reconstructed field with the measured one thus drives the domain evolution by deforming the template domain $\Omega_0$. Conceptually, such a registration map serves as an approximation of the domain gradient, thereby guiding the deformation toward a more accurate domain representation.

To this end, we seek a correction mapping $c$ such that the composed domain $c^{-1} \circ f(\Omega_0)$ more closely approximates the true flow domain. By extending the modeled velocity field $\boldsymbol{u}$ with zero values in the region $D\setminus \Omega$, the refinement is achieved through a registration process that aligns $\boldsymbol{u}$ with the observed data $\tilde{\boldsymbol{u}}$ through a registration procedure by minimizing
\begin{equation}
    \mathbf{L}_\text{reg} = \left\|\tilde{\boldsymbol{u}} - \boldsymbol{u} \circ c \right\|_{L^2(D)}^2,
\label{eq:geometry}
\end{equation}
where we observe that when the correction mapping $c$ is an identity mapping, the optimization functional $\left\|\tilde{\boldsymbol{u}} - \boldsymbol{u} \circ c \right\|_{L^2(D)}$ reduces to the original data fidelity functional in $\|\tilde{\boldsymbol{u}} - \boldsymbol{u}\|_{L^2(D)}$. This reflects the ideal scenario in which no further geometric correction is necessary, indicating that the current domain accurately captures the true vascular region where the flow occurs. The theoretical validation of this domain evolution formulation, that it admits a solution minimizing the original objective \eqref{eq:opt_reform} subject to the Navier–Stokes constraints in \eqref{eq:ns_reform}, will be presented later in this section.

Both the correction mapping and the deformation mapping are modeled as quasi-conformal mappings that are diffeomorphic. This design is essential to preserve the topology of the reference domain, ensuring that $\Omega = f(\Omega_0)$ retains the same topological structure as $\Omega_0$. This property allows the model to incorporate morphological priors effectively. Furthermore, this formulation is central to the theoretical analysis presented later.
We give some detailed explanations of the Beltrami mapping we shall use for the deformation in the following. 
\begin{definition}[Quasi-conformal map]
A quasi-conformal map is a map $f: \mathbb{C} \rightarrow \mathbb{C}$ that satisfies the Beltrami equation
\begin{equation}
\frac{\partial f}{\partial \bar{z}}=\mu(z) \frac{\partial f}{\partial z}
\label{eq:beleq}
\end{equation}
for some complex-valued function called Beltrami coefficient $\mu$ satisfying $\|\mu\|_{\infty}<1$ and such that $\frac{\partial f}{\partial z}$ is non-vanishing almost everywhere. The complex partial derivatives are given by
\begin{equation*}
\frac{\partial f}{\partial z}:=\frac{1}{2}\left(\frac{\partial f}{\partial x}-i \frac{\partial f}{\partial y}\right) 
\quad \text{ and } \quad 
\frac{\partial f}{\partial \bar{z}}:=\frac{1}{2}\left(\frac{\partial f}{\partial x}+i \frac{\partial f}{\partial y}\right).
\end{equation*}
\label{def:qc}
\end{definition}
The Beltrami coefficient $\mu$ quantifies the deviation from conformality. If $\mu = 0$ at a point $p$, the map is conformal in a neighborhood around $p$, reducing \eqref{eq:beleq} to the Cauchy–Riemann equations. In this case, $f$ near $p$ can be written as
\begin{equation*}
f(z) = f(p) + f_z(p)(z + \mu(p)\bar{z}),
\end{equation*}
where $f(p)$ is a translation and $f_z(p)$ a conformal scaling. Thus, $\mu$ captures the non-conformal part of $f$, and $f$ is conformal if and only if $\mu \equiv 0$.

Using a quasi-conformal mapping to represent the deformations of the domain, the geometry subproblem can be finally formulated with respect to a correction mapping $c$ as
\begin{equation}
\min\limits_{\boldsymbol{c}} \quad \mathbf{L}_\text{reg} = \|\tilde{\boldsymbol{u}} - \boldsymbol{u}\circ c\|_{L^2(D)}^2, \quad
\text{s.t. } \quad \mu(c) < 1.
\label{eq:minimize_geometry_2}
\end{equation} 

In our formulation, the quasi-conformal map $c$ effectively serves as a domain gradient, producing a smooth topology-preserving deformation that directly drives boundary evolution. This ensures that the reconstructed geometry keeps the same topology as the template — an important advantage in cardiovascular applications, where vessel networks have consistent anatomical structures and the template provides a meaningful morphological prior. Even in rare cases with atypical anatomy, a coarse template can still be extracted (e.g., from a rough segmentation) and subsequently refined via flow-driven registration. Such topology preservation cannot be achieved by phase-based approaches, and level-set methods lack a theoretical foundation for controlling it. Moreover, because the deformation is obtained by minimizing a global registration energy over the entire flow field, the resulting geometry evolution is smooth, stable, and globally coherent. This stands in contrast to the level-set approach, whose evolution is fundamentally local, driven by contour curvature and neighborhood features, making it difficult to enforce global structure and long-range consistency during optimization.

\subsection{Subproblem Formulations Verification}

After decomposing the original problem into the fluid and geometry subproblems, we now describe the alternating procedure used to iteratively solve these subproblems in order to approximate the solution of the original problem.

Let $n$ denote the iteration number. Given a $f_n$, we first solve for iteration $n$ by minimizing the data discrepancy objective \eqref{eq:fluid} subject to the Navier–Stokes constraints \eqref{eq:ns_reform} over the domain $f_n(\Omega_0)$, where the first mapping $f_1$ is initialized (identity mapping in this work). With the obtained $u_n, p_n, g_n$ from the fluid subproblem (\ref{eq:fluid}),  we then solve the geometry subproblem \ref{eq:geometry} for the correction map to get its minimizer $c_n$. This map updates the domain for the next iteration via $f_{n+1} = c_n^{-1} \circ f_n$. The process of alternatively solving $u_n, p_n, g_n$ and $c_n$ then repeats until convergence is met. We also refer to Figure~\ref{fig:pipeline} for the pipeline flow.

This alternating scheme captures the interplay between fluid physics and geometric deformation: The reconstructed velocity field informs the domain correction, which in turn influences the next fluid solution. As the iterations progress, both the flow field $u_n$ and the mapping $f_n$ are expected to converge to some limits that align with the measured data, the governing physics, and anatomical constraints.

To support this methodology, we establish a theory that, under proper assumptions, the converging sequence generated by this alternating procedure admits a partial and local minimizer of the original problem. The full proof is provided in the appendix for completeness.

\begin{definition}
    Let $F(x, y)$ be a function defined on a product space $\mathcal{X} \times \mathcal{Y}$. The partial minimization of $F$ with respect to the variable $y$, for a fixed $x \in \mathcal{X}$, is defined as
    \begin{equation*}
        \min_{y \in \mathcal{Y}} F(x, y).
    \end{equation*}
\end{definition}
\begin{theorem}
Let $\mathcal{H}(D)$ be the group of diffeomorphisms on a compact domain $D$. Let $\Omega_0$ be a domain with smooth boundary and $\tilde{u}$ measured data in $D$.  Let $\{(u_n, p_n, f_n, g_n, c_n)\}_{n=1}^\infty$ be a sequence of solutions generated by an iterative inverse Navier–Stokes reconstruction algorithm, which means 
    \begin{itemize}[wide=2em, leftmargin=*, nosep, before=\leavevmode]
        \item $f_1 \in \mathcal{H}(D)$ is given,
        \item $(u_n, p_n, g_n)$ minimize $\|\tilde{u} - u\|^2_{L^2(D)}$ with fixed $f_n \in \mathcal{H}(D)$ for $(u_n,p_n,g_n,f_n)$ satisfying \eqref{eq:ns_reform},
        \item $f_{n+1} = c_n^{-1} \circ f_{n}$, where $c_n = \arg\min_{c \in \mathcal{H}(D)} \| \tilde{u} - u_n \circ c \|_{L^2(D)}$.
    \end{itemize}
    Assume the following:
    \begin{enumerate}[wide=2em, leftmargin=*, nosep, before=\leavevmode,label=$(\arabic*)$]
        \item The velocity fields converge: $u_n \to u_*$ in $L^{\infty}(D)$,
        \item The domain mappings converge: $f_n \to f_*$ in $L^{\infty}(D)$,
        \item In a neighborhood of $f_*$ in $L^{\infty}(D)$, minimizing $\|\tilde{u} - u\|^2_{L^2(D)}$ under \eqref{eq:ns_reform} admits a unique solution $(u_n, p_n, g_n)$ and $u \in W^{1,\infty}(D)$. The solution operator $\mathcal{U}: f \mapsto u$ is continuous for the $L^\infty(D)$ norm.
        \item For large $n$, the solution $u_n$ is uniformly Lipschitz.
    \end{enumerate}
    Then, the following conclusions hold:
    \begin{enumerate}[wide=2em, leftmargin=*, nosep, before=\leavevmode,label=$(\alph*)$]
        \item The correction mappings converge: $c_n \to \boldsymbol{id}$ in $L^{\infty}(D)$,
        \item The limit $f_*$ is the partial and local minimizer under \eqref{eq:ns_reform} for $f\in\mathcal{H}(D)$,
        \item The limit $u_*$ is the partial and local minimizer under \eqref{eq:ns_reform} for $f\in\mathcal{H}(D)$,
    \end{enumerate}
    \label{thm:main}
\end{theorem}

Note that assumptions (1)-(2) stem from empirical observations on our numerical scheme. For all cases presented in this paper, the iterative scheme converges in a manner consistent with these assumptions. The uniqueness and the regularity of the solution to the Navier-Stokes equations (i.e., equation~\eqref{eq:ns_reform} with an appropriate choice of boundary data $\boldsymbol{g}$) are established in \cite[Theorem 3.7.3]{sohr2012navier} and \cite[Theorem 7.1]{kelliher2006navier}. Therefore, given proper smoothness of $\Omega_0$ and $f_n$, assuming the uniqueness of the associated minimization problem and the uniform Lipschitz continuity of the solution sequence is reasonable and consistent with the well-posedness of the PDE constraint.
\section{Numerical Method}

\begin{figure}[t!]
    \centering
    \includegraphics[width=\linewidth]{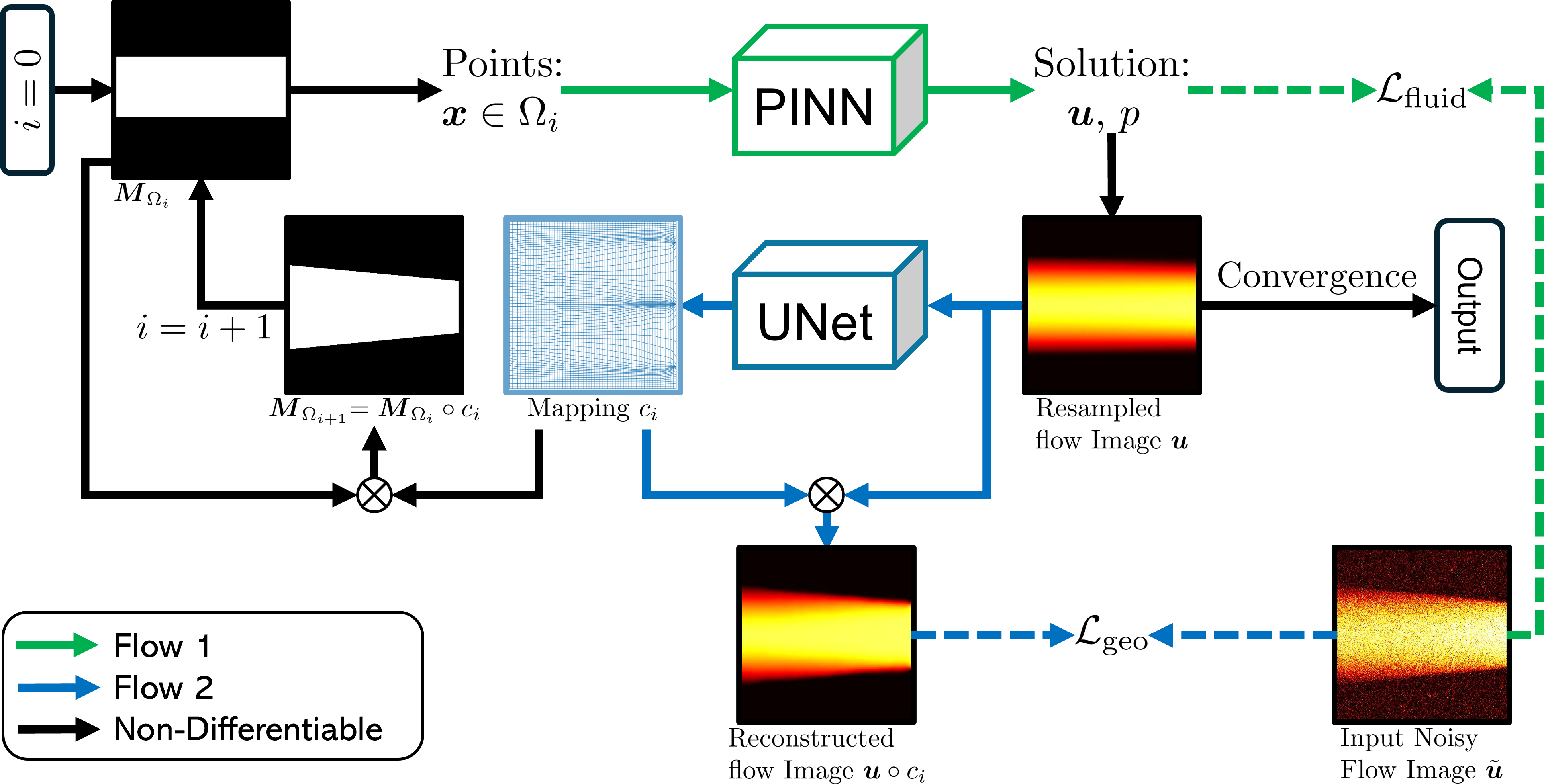}
    \caption{Overview of the iterative pipeline: starting from an initial mask defining the flow domain, points are sampled and input to a PINN for flow prediction. The predicted flow is resampled over the image domain, followed by quasi-conformal mapping to align with the noisy image. The mask is updated via this mapping, and the process repeats until convergence.}
    \label{fig:pipeline}
\end{figure}

In the following sections, we provide specific implementation details, including the point sampling strategy within the region and along the boundaries defined by the template mask, the formulation of the fluid subproblem, and the geometry modeling. Finally, we present the complete algorithm.

\subsection{Mask Sampling for Flow Domain}
\label{sec:method_sampling}
Since a binary mask is used to represent the flow region where blood is expected to circulate, special consideration is required when sampling points for PINN training. We now describe our sampling strategy for the fluid domain $\Omega$ and its boundary $\partial\Omega$ at iteration $n$. Denoting the deformation mapping at this iteration by $f$, the domain is given by $\Omega = f(\Omega_{0})$. Similarly, as in \eqref{eq:mapping_boundary}, the boundary components satisfy
\begin{equation}
\Gamma^\text{i} = f(\Gamma_{0}^\text{i}), \quad 
\Gamma^\text{o} = f(\Gamma_{0}^\text{o}), \quad 
\Gamma^\text{w} = f(\Gamma_{0}^\text{w}),
\end{equation}
The input data consists of a velocity image $\tilde{\boldsymbol{u}}$ that is defined on $[0, \varepsilon H] \times [0, \varepsilon K]$, where $\varepsilon$ is determined by the image resolution. Each pixel point is positioned at coordinates
\begin{equation}
    \mathcal{D} \coloneq \{\left( \varepsilon / 2 + h \cdot \varepsilon, \; \varepsilon / 2 + k \cdot \varepsilon \right) \text{ : where $h=0,\dots,H-1$ and $k=0,\dots,K-1$}\}.
    \label{eq:grid}
\end{equation}
Note that although the value for $\tilde{\boldsymbol{u}}$ is only explicitly defined on the grid points in $\mathcal{D}$, its values at other locations can be obtained through interpolation methods; in this work, bilinear interpolation is used for this purpose.

Let us also denote the mask that is used to represent the domain $\Omega$ as $M_{\Omega}$ defined on pixel points $\mathcal{D}$ as
\begin{equation}
    \boldsymbol{M}_{\Omega} =
    \begin{cases}
        1 \quad \text{for $\boldsymbol{x}\in \mathcal{W}$},\\
        0 \quad \text{for $\boldsymbol{x}\in \mathcal{D}\setminus\mathcal{W}$}.
    \end{cases}
    \label{eq:mask}
\end{equation}
where $\mathcal{W} = \mathcal{D}\cap\Omega$ are the pixel points within the domain $\Omega$. Using an edge detector \cite{parker2010algorithms} on $\boldsymbol{M}_{\Omega}$, which is enforced to be 4-connected, we could then obtain the edge image as $\boldsymbol{M}_{\Gamma}$ illustrated in Figure~\ref{fig:Illustration_boundary}.(a). The point set that was used to approximately represent the boundary of $\Omega$ is as 
\begin{equation*}
    \mathcal{E} = \{ \boldsymbol{x} \in \mathcal{D} : \boldsymbol{M}_{\Gamma}(\boldsymbol{x})=1\} \in \mathcal{G} .
\end{equation*}

\begin{figure}[t!]
    \centering
    \includegraphics[width=\linewidth]{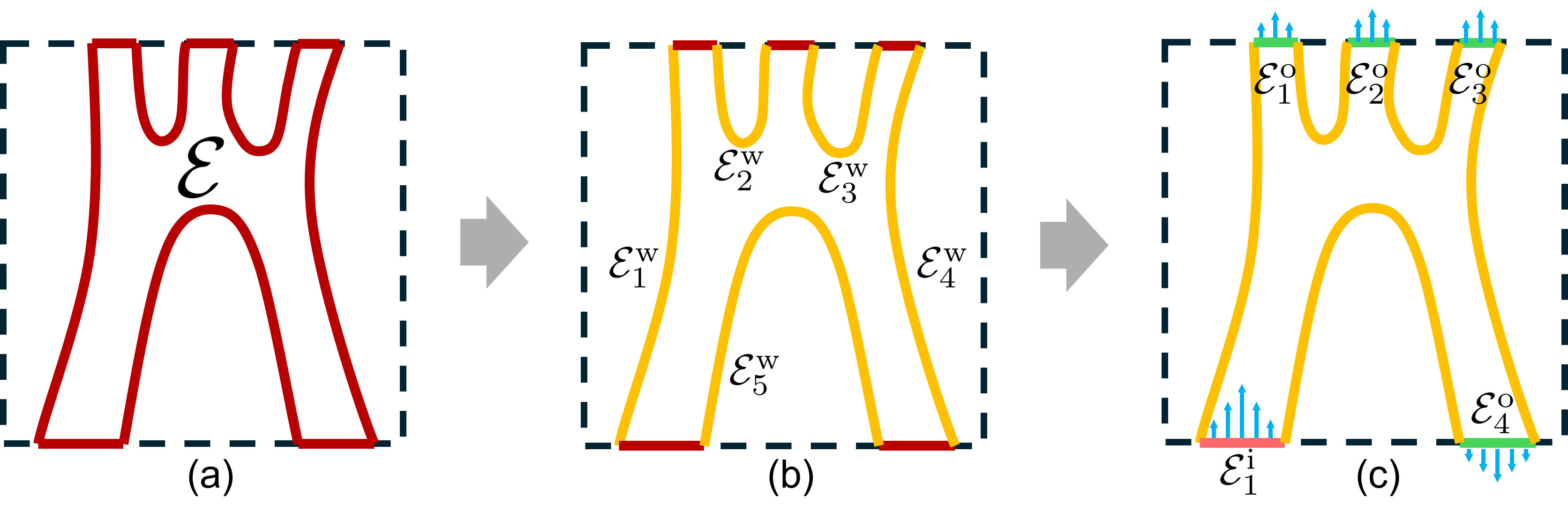}
    \caption{Illustration of the steps used to divide the entire boundary into wall, outlet, and inlet boundaries. Red indicates the full boundary or the unclassified segments; yellow denotes wall boundaries; green represents outlets; and pink indicates inlets.}
    \label{fig:Illustration_boundary}
\end{figure}

The inlet and outlet boundaries are constrained to lie on the edges of the image grid, with coordinates either at $\tfrac{1}{2}\varepsilon$ or $H\varepsilon - \tfrac{1}{2}\varepsilon$. In contrast, the vessel wall boundary can be represented by a set of points that are not on the grid boundary as follows
\begin{equation}
    \mathcal{E}^\text{w} = \{(x_1, x_2) \in \mathcal{E} \;|\; x_1 \notin \{\tfrac{1}{2}\varepsilon, H\varepsilon - \tfrac{1}{2}\varepsilon\};\; x_2 \notin \{\tfrac{1}{2}\varepsilon, K\varepsilon - \tfrac{1}{2}\varepsilon\} \},
\end{equation}
and divide $\mathcal{E}^\text{w}$ into $Z_\text{w}$ disjoint points collections as
\begin{equation*}
    \mathcal{E}^\text{w} = \mathcal{E}^\text{w}_1 \cup \mathcal{E}^\text{w}_2 \cup \dots \mathcal{E}^\text{w}_{Z_\text{w}},
\end{equation*}
where the points of each $\mathcal{E}^\text{w}_z$ represent a curve in the image that is 4-connected, and each set are mutually disjoint.

Furthermore, we also divide $\mathcal{E}\setminus\mathcal{E}^\text{w}$ into multiple disjoint points collection as
\begin{equation*}
\begin{aligned}
    \mathcal{E}^\text{i} &=\mathcal{E}^\text{i}_1 \cup \mathcal{E}^\text{i}_2 \cup \dots \cup \mathcal{E}^\text{i}_{Z_\text{i}},\\
    \mathcal{E}^\text{o} &=\mathcal{E}^\text{o}_1 \cup \mathcal{E}^\text{o}_2 \cup \dots \cup \mathcal{E}^\text{o}_{Z_\text{o}},
\end{aligned}
\end{equation*}
where each of the point sets is 4-connected, and the sets are mutually disjoint. To determine whether a given point collection $\mathcal{E}^*_z$ (where $*\in \{\text{i},\text{o}\}$ corresponds to an inlet or outlet), we first compute the average velocity over the points in the set as follows:
\begin{equation}
    \bar{\boldsymbol{u}}^* = \tfrac{1}{|\mathcal{E}^\text{*}|}\sum\limits_{\boldsymbol{x}\in\mathcal{E}^\text{*}} \boldsymbol{u}(\boldsymbol{x}).
\end{equation}
where $|\mathcal{E}^\text{*}_z|$ represent the number of points in $\mathcal{E}^\text{*}_z$. Then, as in Figure~\ref{fig:Illustration_boundary}.(c), if $\bar{\boldsymbol{u}}^*$ points inward to the image domain, which indicates an inflow pattern, $\mathcal{E}^\text{*}$ should be an inlet boundary. Otherwise, the average velocity points outward from the image domain, indicating an outflow pattern, and thus the collection corresponds to an outlet.

Thus, the pixel points for the interior region of $\Omega$ is as
\begin{equation*}
    \mathcal{G} = \mathcal{W} - \mathcal{E}.
\end{equation*}
and with $\mathcal{G}$ and $\mathcal{E}$ in hand, we can then do sampling for the interior domain $\Omega \setminus\partial\Omega$ and the boundary $\Gamma$.

To sample a new point in the interior domain $\Omega \setminus\partial\Omega$, we could first randomly choose one point $\boldsymbol{p}$ from $\mathcal{G}$. Then, the sampled point is as
\begin{equation}
    \boldsymbol{x}^\text{g} = \boldsymbol{p} + [l_1,l_2]',
    \label{eq:sample_domain}
\end{equation}
where $l_1$ and $l_2$ are randomly sampled from $(-\epsilon,\epsilon)$.

To sample points along a segment of $\mathcal{E}^\text{i}_z$, $\mathcal{E}^\text{o}_z$, and $\mathcal{E}^\text{w}_z$, we apply piecewise linear interpolation. Specifically, let $\mathcal{E}^*$ denote one such segment consisting of an ordered set of points $\{ \boldsymbol{p}_0, \boldsymbol{p}_1, \ldots, \boldsymbol{p}_{N_*} \}$, where each consecutive pair $(\boldsymbol{p}_j, \boldsymbol{p}_{j+1})$ defines a straight-line segment. Then, we first generate a random value $l \in [0, L]$, where $L$ is the total length of the boundary. We then determine which segment the sampled point will lie in by accumulating in each segment's length $S_r$ until $L_j = \sum_{r=1}^j S_r \geq l$ is met, which means the sampled points within segment $j$ as $(p_{j-1},p_j)$. We finally have the sampled point on the boundary as
\begin{equation}
    \boldsymbol{x}^\text{e} = \frac{L_j-l}{L_j - L_{j-1}} \boldsymbol{p}_{j-1} + \frac{l-L_{j-1}}{L_j - L_{j-1}} \boldsymbol{p}_{j}.
    \label{eq:sample_boundary}
\end{equation}

\subsection{Fluid Subproblem}
\label{sec:method_fluid}

\begin{figure}[t!]
    \centering
    \includegraphics[width=\linewidth]{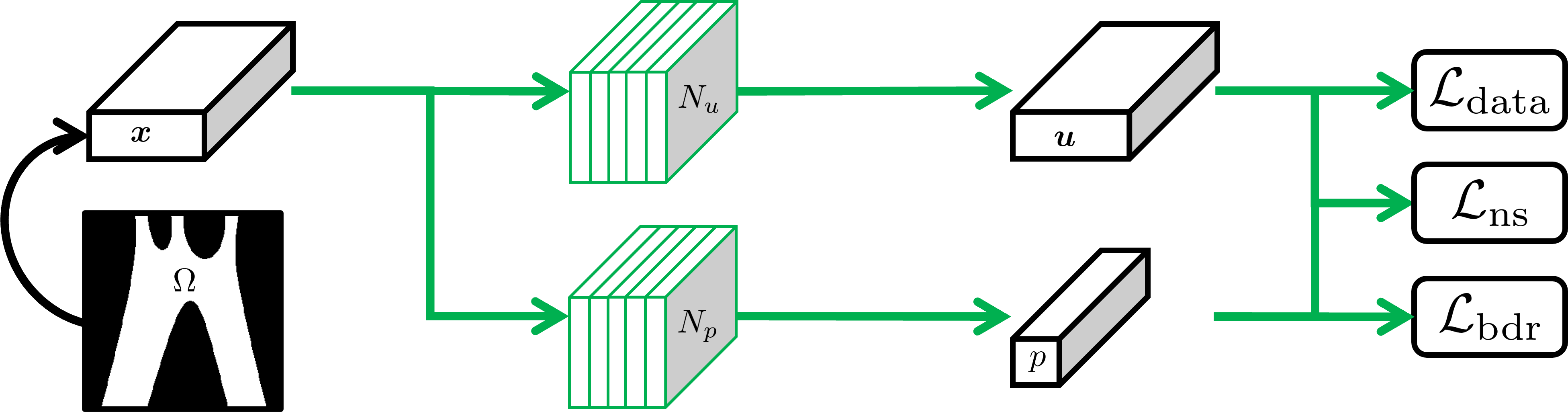}
    \caption{Architecture of the PINN for fluid flow prediction, featuring two separate networks dedicated to velocity and pressure estimation.}
    \label{fig:network_fluid}
\end{figure}

We describe some details for optimizing the fluid subproblems \eqref{eq:fluid} for each iteration $n$ through Physics-Informed Neural Networks. For convenience, we denote the current spatial domain as $\Omega$, omitting the iteration index. This represents $\Omega_n$ at iteration $n$. We first randomly sample $M^\text{g}$ points from the inner domain as $\boldsymbol{x}^\text{g}_m\in\Omega \setminus\partial\Omega$ according to \eqref{eq:sample_domain} for $m=1,\dots,M^\text{g}$. Regarding the boundaries, we randomly sample the coordinates $M^\text{i}$, $M^\text{o}$ and $M^\text{w}$ at the inlet, outlet and wall boundary as $\boldsymbol{x}^\text{i}_m\in\Gamma^\text{i}$, $\boldsymbol{x}^\text{o}_m\in\Gamma^\text{o}$ and $\boldsymbol{x}^\text{w}_m\in\Gamma^\text{w}$ according to \eqref{eq:sample_boundary}, respectively. We then introduce two networks as given in Figure~\ref{fig:network_fluid}: one for the velocity computation, named $\mathcal{N}_u$, and another for the pressure computation, named $\mathcal{N}_p$. We use $\boldsymbol{\theta}_{u}$ and $\boldsymbol{\theta}_{p}$ the trainable parameters for these two neural networks. The two networks map any point $\boldsymbol{x}\in \Omega$ to the values of the velocity and pressure at this point, respectively. Omitting thesubscripts of the sampled points, the flow velocity and pressure for the point $\boldsymbol{x}^\text{g}$ can be calculated as
\begin{equation}
\begin{aligned}
    \boldsymbol{u}^\text{g}(\boldsymbol{x}^\text{g}; {\boldsymbol{\theta}}_{u}) &= \mathcal{N}_u(\boldsymbol{x}^\text{g}; {\boldsymbol{\theta}}_{u}),\\
    p^\text{g}(\boldsymbol{x}^\text{g};  \boldsymbol{\theta}_{p}) &= \mathcal{N}_p(\boldsymbol{x}^\text{g}; {\boldsymbol{\theta}}_{p}),
\end{aligned} 
 \label{eq:fluidnetworkoutput}
\end{equation}
where for the boundaries, the notations are changed correspondingly as $\boldsymbol{u}^\text{i}$, $\boldsymbol{u}^\text{o}$, $\boldsymbol{u}^\text{w}$ and $p^\text{o}$. Also, let's define the discretized norm used throughout. Let $U$ denote one of the relevant domains: $\Omega$, $\Gamma^\text{i}$, $\Gamma^\text{o}$, or $\Gamma^\text{w}$. For each domain $U$, let $\{\boldsymbol{x}_m\}_{m=1}^M$ represent the corresponding set of sampled points (indexed appropriately). Given a function $F : U \to \mathbb{R}$, we define the discrete $L^2$ norm as
\begin{equation}
\|F\|_{L^2(U;M)} \coloneq \left( \frac{1}{M} \sum_{m=1}^M |F(\boldsymbol{x}_m)|^2 \right)^{1/2} \quad \text{ where } x_m \sim U.
\end{equation}

{To train the networks, in addition to the measured data, the governing partial differential equations should also be incorporated into the loss function through their residuals to make the neural networks physics-informed. Since the networks take spatial coordinates as input and predict physical quantities such as velocity and pressure, automatic differentiation (e.g., in PyTorch or TensorFlow) can be used to compute the derivatives required to evaluate these residuals. We then present the details of these residuals, which together constitute the loss function.}

As the modeled velocity is extended by zero value over $D\setminus\Omega$, the data fitting term is thus only defined in the domain $\Omega$ as
\begin{equation}
\mathcal{L}_{\text{data}}(\boldsymbol{\theta}_{u} , \boldsymbol{\theta}_{p}) = \|\tilde{\boldsymbol{u}} - \boldsymbol{u}^\text{g}\|^2_{L^2(\Omega;M^\text{g})}.
\end{equation}
To make the process physics-informed, we rewrite the Navier-Stokes equation into residual forms
\begin{equation}
\begin{aligned}
\mathcal{L}_{\text{ns}}(\boldsymbol{\theta}_{u} , \boldsymbol{\theta}_{p}) &= \|\boldsymbol{u}^\text{g} \cdot \nabla \boldsymbol{u}^\text{g} -  \nabla p^\text{g} + \nu \Delta \boldsymbol{u}^\text{g} \|^2_{L^2(\Omega;M^\text{g})}\\
&+ \|\nabla\cdot\boldsymbol{u}^\text{g}\|^2_{L^2(\Omega;M^\text{g})},\\
\end{aligned}
\label{eq:method_fluid_ns}
\end{equation}
and its associated boundary conditions into residual form as
\begin{equation}
\begin{aligned}
\mathcal{L}_{\text{bdr}}(\boldsymbol{\theta}_{u} , \boldsymbol{\theta}_{p}) 
    =\|\boldsymbol{u}^\text{i} - \boldsymbol{g}^\text{i} \|^2_{L^2(\Gamma^\text{i};M^\text{i})}
    +\|(- p^\text{o} + \nu \nabla \boldsymbol{u}^\text{o})\cdot \boldsymbol{n}  \|^2_{L^2(\Gamma^\text{o};M^\text{o})}
    +\| \boldsymbol{u}^\text{w} \|^2_{L^2(\Gamma^\text{w};M^\text{w})}.
\end{aligned}
\label{eq:method_fluid_bdr}
\end{equation}
We also assume that the inlet velocity boundary follows a parabolic profile. Therefore, the inlet velocity prediction should be regularized to be a parabolic velocity profile whose average velocity matches the average noisy velocity over $\Gamma^\text{i}$, given by
\begin{equation*}
\boldsymbol{g}^\text{r}(\boldsymbol{x}) = \tfrac{3}{2}\tfrac{\int_{\Gamma^\text{i} }\tilde{\boldsymbol{u}}}{|\Gamma^\text{i}|} \left(1 - \tfrac{|\boldsymbol{q} - \boldsymbol{x}|^2}{(|\Gamma^\text{i}| / 2)^2}\right),
\end{equation*}
where $\boldsymbol{q}$ is the centroid of $\Gamma^\text{i}$ and $|\Gamma^\text{i}|$ represents the length of it. This leads to an additional regularization for the inlet velocity as
\begin{equation}
    \mathcal{L}_\text{in}    = \|\boldsymbol{u}^\text{i} - \boldsymbol{g^\text{r}}\|_{L^2(\Gamma^\text{i};M^\text{i})}^2,
\label{eq:method_fluid_gr}
\end{equation}
Integrating the losses from both the flow and deformation problems with proper boundary losses, we obtain
\begin{equation}
\mathcal{L}_{\text{fluid}}(\boldsymbol{\theta}_{u} , \boldsymbol{\theta}_{p}) = \alpha_{\text{data}} \mathcal{L}_{\text{data}} + \alpha_{\text{ns}} \mathcal{L}_{\text{ns}} + \alpha_{\text{bdr}} \mathcal{L}_{\text{bdr}} + \alpha_{\text{in}} \mathcal{L}_\text{in},
\label{eq:method_fluid_losses}
\end{equation}
where $\alpha_{\text{data}}$, $\alpha_{\text{ns}}$, $\alpha_{\text{bdr}}$ and $\alpha_{\text{in}}$ are the weights for reconstruction error, the Navier-Stokes equations, the boundary condition and inlet velocity regularization respectively. Then, the fluid optimization problem is modeled as follows,
\begin{equation}
    \min _{\boldsymbol{\theta}_{u} , \boldsymbol{\theta}_{p}} \mathcal{L}_{\text{fluid}}(\boldsymbol{\theta}_{u} , \boldsymbol{\theta}_{p}). 
    \label{eq:method_fluid_minimization}
\end{equation}

To reduce the loss values, gradient descent methods are used to update $\boldsymbol{\theta}_{u}$ and $\boldsymbol{\theta}_{p}$ with the help of two distinct optimizers by
\begin{equation}
\begin{aligned}
&\boldsymbol{\theta}_{u} \leftarrow \boldsymbol{\theta}_{u} - \tau_{u} \nabla_{\boldsymbol{\theta}_{u}} \mathcal{L}_{\text{fluid}}(\boldsymbol{\theta}_{u} , \boldsymbol{\theta}_{p}),\\
&\boldsymbol{\theta}_{p} \leftarrow \boldsymbol{\theta}_{p} - \tau_{p} \nabla_{\boldsymbol{\theta}_{p}} \mathcal{L}_{\text{fluid}}(\boldsymbol{\theta}_{u} , \boldsymbol{\theta}_{p}).
\end{aligned}
\label{eq:optimizer_fluid}
\end{equation}
where $\nabla_{\boldsymbol{\theta}_{u}}$ and $\nabla_{\boldsymbol{\theta}_{p}}$ are the gradients with respect to $\boldsymbol{\theta}_{u}$ and $\boldsymbol{\theta}_{p}$, which are computed by autogradient within \textit{PyTorch}. Here, $\tau_{u}$ and $\tau_{p}$ represent adaptive step sizes determined by various methods \cite{amari1993backpropagation, KingBa15}, with step sizes computed according to the approach outlined in \cite{KingBa15}.

After sufficiently training the neural networks $\mathcal{N}_u$ and $\mathcal{N}_p$ in iteration $n$ by minimizing \eqref{eq:method_fluid_losses} with $\Omega = \Omega_n$ in \eqref{eq:method_fluid_ns}--\eqref{eq:method_fluid_gr}, we sample the reconstructed flow image over the entire image grid $\mathcal{D}$ (defined in \eqref{eq:grid} for use in the geometry subproblem to be detailed in the next subsection. The sampled velocity field $\hat{\boldsymbol{u}}$ is defined as:
\begin{equation}
\begin{aligned}
    \hat{\boldsymbol{u}}(\boldsymbol{x}) = \begin{cases}
        \boldsymbol{u}^\text{g}(\boldsymbol{x}), & \text{if } \boldsymbol{x} \in \mathcal{G}, \\
        \boldsymbol{u}^\text{i}(\boldsymbol{x}), & \text{if } \boldsymbol{x} \in \mathcal{E}^\text{i}, \\
        \boldsymbol{u}^\text{o}(\boldsymbol{x}), & \text{if } \boldsymbol{x} \in \mathcal{E}^\text{o}, \\
        \boldsymbol{u}^\text{w}(\boldsymbol{x}), & \text{if } \boldsymbol{x} \in \mathcal{E}^\text{w}, \\
        0, & \text{otherwise}.
    \end{cases}
\end{aligned}
\label{eq:un}
\end{equation}
Here, $\mathcal{G}$, $\mathcal{E}^\text{i}$, $\mathcal{E}^\text{o}$, and $\mathcal{E}^\text{w}$ denote the sets of points in the discretized image grid corresponding to $\Omega$, $\Gamma^\text{i}$, $\Gamma^\text{o}$, and $\Gamma^\text{w}$ as described in Section~\ref{sec:method_sampling}, respectively.

\subsection{Geometry Subproblem}

\begin{figure}[t!]
    \centering
    \includegraphics[width=0.85\linewidth]{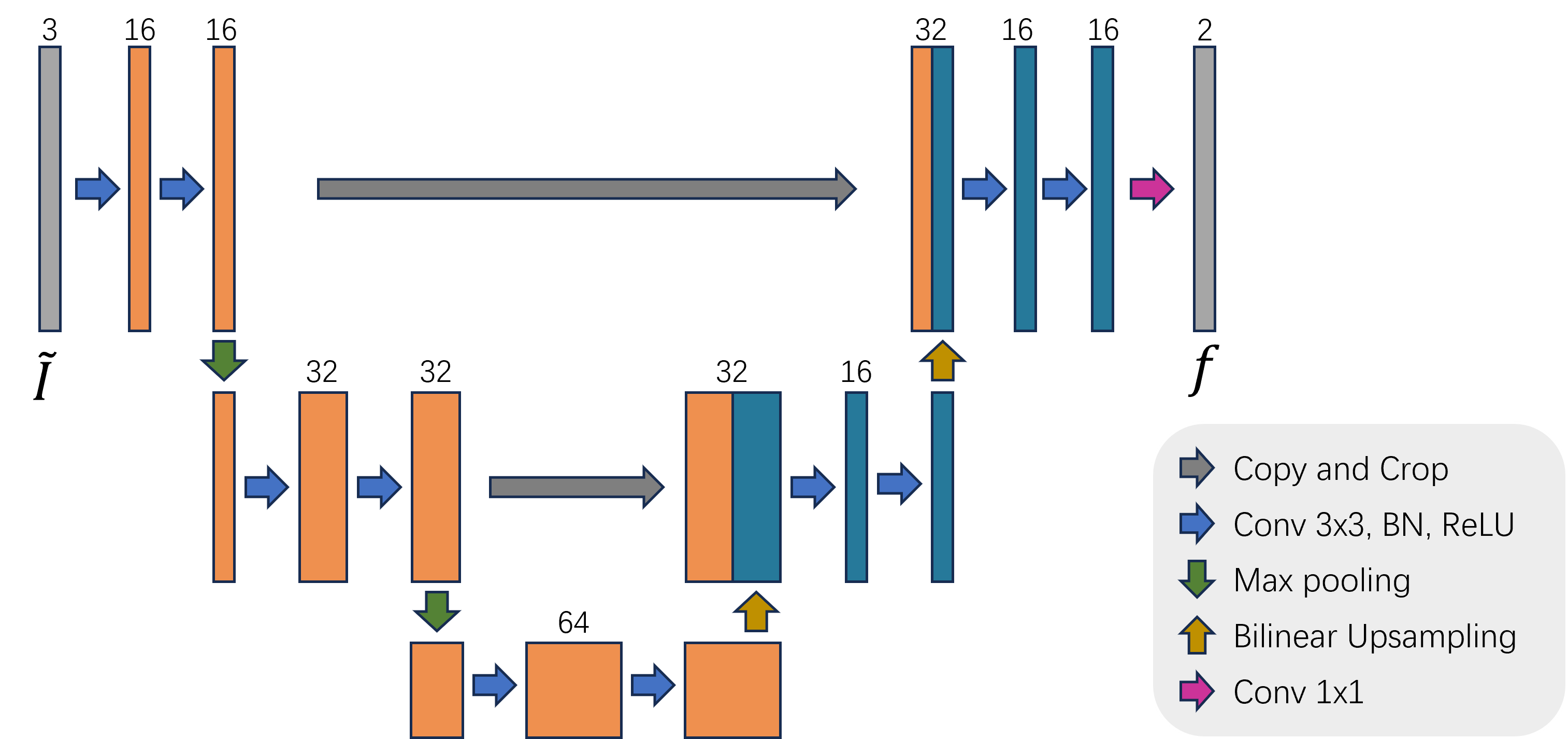}
    \caption{Architecture of the UNet used for quasi-conformal mapping estimation.}
    \label{fig:QCE}
\end{figure}

We now address the geometry subproblem for each iteration $n$. We first obtain the value of solved $\hat{\boldsymbol{u}}$ from the fluid subproblem. Then, we also define the norm over the discretized grid $\mathcal{D}$ for a function $F : \mathcal{D} \to \mathbb{R}$ as
\begin{equation}
\|F\|_{L^2(\mathcal{D})} = \left( \frac{1}{HK}\sum_{\boldsymbol{x} \in \mathcal{D}} |F(\boldsymbol{x})|^2 \right)^{1/2},
\end{equation}
As discussed in the previous section, we need to register the image $\hat{\boldsymbol{u}}$ obtained from the solution of the fluid subproblem \eqref{eq:method_fluid_minimization} to the captured noisy image $\tilde{\boldsymbol{u}}$ to produce a mapping $c$, which serves as an approximation of the domain gradient. This mapping is generated by a neural network
\begin{equation}
    c = \mathcal{N}_c(\hat{\boldsymbol{u}};\boldsymbol{\theta}_c)
\end{equation}
where $\boldsymbol{\theta}_c$ is the trainable weights of the network. Following the approach of \cite{zhang2026quasi,zhang2024learning}, we use a U-Net architecture for $\mathcal{N}_c$, as illustrated in Figure~\ref{fig:QCE}, to estimate the quasi-conformal mapping. To ensure accurate registration, the network is trained by minimizing the following registration loss.
\begin{equation}\label{eq:method_geometry_register}
    \mathcal{L}_{\text{reg}}(\boldsymbol{\theta}_c) = \|\tilde{\boldsymbol{u}} - \hat{\boldsymbol{u}}\circ c\|^2_{L^2(\mathcal{D})},
\end{equation}
For $c$ to be a quasi-conformal mapping, its associated Beltrami coefficient must be strictly less than 1 according to Definition~\ref{def:qc}. In addition, to promote the mapping $c$ to be diffeomorphic, we apply Laplacian regularization to encourage smoothness and control geometric distortion. Thereby, the regularization terms are introduced as
\begin{equation}\label{eq:losses}
\begin{aligned}
    \mathcal{L}_{\text{bc}}(\boldsymbol{\theta}_c) &= \|e^{\mu(c)-1}\|^2_{L^2(\mathcal{D})}, \\
    \mathcal{L}_{\text{lap}}(\boldsymbol{\theta}_c) &= \|\Delta c\|^2_{L^2(\mathcal{D})},
\end{aligned}
\end{equation}
where $\mu(c)$ represents the Beltrami coefficient of the mapping $c$, computed through a finite difference  implementation of Equation \eqref{eq:beleq}. The term $\Delta c$ denotes the Laplacian of the mapping $c$. Together, these regularizations ensure that $c$ adheres to quasi-conformal constraints while maintaining desirable geometric properties, such as continuity and smoothness.

The overall objective function for the geometry subproblems \eqref{eq:geometry} is then replaced by
\begin{equation}
\mathcal{L}_{\text{geo}}(\boldsymbol{\theta}_c) = \alpha_{\text{reg}} \mathcal{L}_{\text{reg}}
    +\alpha_{\text{bc}} \mathcal{L}_{\text{bc}}
    +\alpha_{\text{lap}} \mathcal{L}_{\text{lap}},
    \label{eq:method_geometry_losses}
\end{equation}
where $\alpha_{\text{reg}}$, $\alpha_{\text{bc}}$, and $\alpha_{\text{lap}}$ are the weighting parameters for the registration error, Beltrami regularization, and Laplacian regularization, respectively. Then, the geometry subproblem is reformulated as
\begin{equation}
    \min _{\boldsymbol{\theta}_c} \mathcal{L}_{\text{geo}}(\boldsymbol{\theta}_c). 
    \label{eq:method_geometry_minimization}
\end{equation}
To reduce the loss values, gradient descent methods are used to update $\boldsymbol{\theta}_{c}$ by
\begin{equation}
\boldsymbol{\theta}_{c} \leftarrow \boldsymbol{\theta}_{c} - \tau_{c} \nabla_{\boldsymbol{\theta}_{c}} \mathcal{L}_{\text{geo}}(\boldsymbol{\theta}_{c}),
\label{eq:optimizer_geometry}
\end{equation}
where $\nabla_{\boldsymbol{\theta}_{c}}$ is the gradient with respect to $\boldsymbol{\theta}_{c}$, $\tau_{c}$ represents adaptive step sizes as in \cite{amari1993backpropagation}. After sufficient training, the produced mapping is used to warp the current mask, thereby refining the vessel region. This updated domain is then used in the next iteration of fluid subproblem training.

\subsection{Overall Algorithm}

The proposed algorithm iteratively alternates between solving the flow subproblem \eqref{eq:method_fluid_minimization} and the geometry subproblem \eqref{eq:method_geometry_minimization} to achieve accurate blood flow imaging. As a result, the overall optimization problem \eqref{eq:opt_reform}, subject to the constraint in \eqref{eq:ns_reform}, is approximately solved through this alternating scheme between the fluid subproblem \eqref{eq:method_fluid_minimization} and the geometry subproblem \eqref{eq:method_geometry_minimization}.

We begin with the fluid subproblem using the domain mask $M_{\Omega_n}$, which represents the mask at iteration $n$, as defined in \eqref{eq:mask}. The initial mask $M_{\Omega_1}$, provided at the beginning, serves as both a topological prior and a coarse approximation of the spatial domain in which blood flow is expected to occur. Based on the current mask $M_{\Omega_n}$, spatial points are sampled and used as input to the PINN models $\mathcal{N}_u$ and $\mathcal{N}_p$ in \eqref{eq:fluidnetworkoutput} used for \eqref{eq:method_fluid_minimization}. These two networks are then trained on the sampled points by minimizing the fluid loss function defined in \eqref{eq:method_fluid_losses}, using the gradient descent described in \eqref{eq:optimizer_fluid}. Once sufficiently trained, the models produce a velocity field that adheres to the governing physical laws \eqref{eq:ns_reform} while fitting the noisy observational data $\tilde{u}$.

After training, we evaluate the predicted velocity at the points of the image grid $\mathcal{D}$ (as defined in \eqref{eq:grid}), which produces a reconstructed flow image $\hat{\boldsymbol{u}}_n$, which corresponds to the indexed sampled velocity image at iteration $n$, denoted by $\hat{\boldsymbol{u}}$ in \eqref{eq:un}. To correct for potential misalignments between the solved domain of this iteration and the actual vessel geometry, we estimate a quasi-conformal registration mapping that aligns $\hat{\boldsymbol{u}}_n$ with the observed flow $\tilde{\boldsymbol{u}}$. This is achieved by training the mapping network $\mathcal{N}_c$ to minimize the registration loss \eqref{eq:method_geometry_losses}, using the gradient descent scheme in \eqref{eq:optimizer_geometry}. After sufficient training, the network outputs a registration mapping $c_n$, which is used to update the vessel mask by warping the previous domain: $M_{\Omega_{n+1}} = M_{\Omega_n} \circ c_n$. As shown in Remark~\ref{remark:warp}, this is equivalent to $M_{\Omega_{n+1}} = M_{c_n^{-1}(\Omega_n)}$. The updated mask $M_{\Omega_{n+1}}$ then replaces the previous one and serves as the domain mask for the next iteration, as illustrated in Figure~\ref{fig:pipeline}, enabling a new round of point sampling for the fluid subproblem training. 

This alternation between the PINN-based flow reconstruction problem \eqref{eq:method_fluid_minimization} and domain refinement through the quasi-conformal registration problem \eqref{eq:method_geometry_minimization}  continues iteratively. The algorithm terminates when either a convergence criterion is satisfied or the maximum number of iterations is reached. The overall algorithm is summarized in Algorithm \ref{alg:overall} for a clear presentation.

\begin{remark}\label{remark:warp}
The two formulations
\begin{equation}
M_{\Omega_{n+1}} = M_{c_n^{-1}(\Omega_n)} \quad \text{and} \quad M_{\Omega_{n+1}} = M_{\Omega_n} \circ c_n
\end{equation}
are equivalent under the interpretation of $M_{\Omega}$ as the characteristic function of the domain $\Omega$. This equivalence is seen by evaluating the composition:
\begin{equation}
M_{\Omega_n} \circ c_n(x) = 1 
\quad \Longleftrightarrow \quad c_n(x) \in \Omega_n 
\quad \Longleftrightarrow \quad x \in c_n^{-1}(\Omega_n).
\end{equation}
Therefore, $M_{\Omega_n} \circ c_n$ is the characteristic function of $c_n^{-1}(\Omega_n)$, which justifies the update $M_{\Omega_{n+1}} = M_{\Omega_n} \circ c_n$ as a computationally convenient form to generate the warped mask. 
\end{remark}

\begin{algorithm}
\caption{The main algorithm}
\label{alg:overall}
\begin{algorithmic}[1]
\Require Measured flow image $\tilde{\boldsymbol{u}}$, reference mask $M_{\Omega_0}$;
\Ensure Reconstructed flow image $\boldsymbol{u}$, reconstructed mask $M_{\Omega}$;
\State $M_{\Omega_1} \gets M_{\Omega_{0}}$;
\State $\boldsymbol{\theta}_u,\boldsymbol{\theta}_p,\boldsymbol{\theta}_c \gets$ random initialization;
\For{$n=1:N_\text{MAX}$}
    \State Sample points using $M_{\Omega_n}$ as described in Section \ref{sec:method_sampling};
    \State $\boldsymbol{\theta}_u,\boldsymbol{\theta}_p \gets $ gradient descent as \eqref{eq:optimizer_fluid} with the sampled points;
    \For{each $\boldsymbol{x}\in\mathcal{D}$ (defined in \eqref{eq:grid})}
        \State $\hat{\boldsymbol{u}}(\boldsymbol{x}) \gets \mathcal{N}_u(\boldsymbol{x};\boldsymbol{\theta}_u)$ ;
    \EndFor
    \State $\hat{\boldsymbol{u}}_n \gets \hat{\boldsymbol{u}}$;    
    \If{$\|\tilde{\boldsymbol{u}}-\hat{\boldsymbol{u}}_n\|^2_{L^2(\mathcal{D})}$ converge}
        \State $\boldsymbol{u} \gets \hat{\boldsymbol{u}}_n$, $M_{\Omega} \gets M_{\Omega_n}$;
        \State Finish and output;
    \EndIf
    \State $\boldsymbol{\theta}_c \gets $ gradient descent as \eqref{eq:optimizer_geometry} with $\hat{\boldsymbol{u}}$;
    \State $c_n \gets \mathcal{N}_c(\hat{\boldsymbol{u}}_n;\boldsymbol{\theta}_c)$;
    \State $M_{\Omega_{n+1}} \gets M_{\Omega_n} \circ c_n$;
\EndFor
\end{algorithmic}
\end{algorithm}
\section{Experiment}

\subsection{Implementation Details}\,

\textbf{Data Synthesis}
Synthetic data were generated using MATLAB FEATool with the  inlet velocity profile
\begin{equation}
\boldsymbol{g}(\boldsymbol{x}) = v \left(1 - \tfrac{|\boldsymbol{q} - \boldsymbol{x}|^2}{(|\Gamma^\text{i}| / 2)^2}\right),
\end{equation}
where $v$ is a velocity parameter chosen for a specific experiment, $\boldsymbol{q}$ denotes the centroid of the inlet boundary $\Gamma^\text{i}$, and $|\Gamma^\text{i}|$ represents its length (in 2D). All synthesized images were of size $256 \times 256$. The spatial resolution for each experiment is reported individually.

\textbf{Noise Model}
To simulate the measurement noise in MRI flow imaging, we adopted a signal-dependent noise model combining Gaussian components
\begin{equation}
\tilde{\boldsymbol{u}} = \boldsymbol{u} + \sqrt{|\boldsymbol{u}|} \cdot \boldsymbol{n}_1 + \boldsymbol{n}_2,
\label{eq:noise_model}
\end{equation}
where $\boldsymbol{n}_1 \sim \mathcal{N}(0, \sigma_1^2)$ and $\boldsymbol{n}_2 \sim \mathcal{N}(0, \sigma_2^2)$ are two independent Gaussian noise terms.

\textbf{Network Architecture}
The PINN model consists of two separate MLPs for velocity and pressure. The velocity network has 8 layers with 40 neurons per layer, while the pressure network also has 8 layers with 20 neurons per layer. Increasing the number of layers and neuros can improve expressiveness, but can lead to a higher computational cost without significant performance gains \cite{zhang2024meshless}. The U-Net architecture used to generate the correction mapping is shown in Figure \ref{fig:QCE} and has been demonstrated to be effective for producing high-quality quasi-conformal mappings in previous works \cite{zhang2021topology,zhang2025deformation}.

\textbf{Error Metric}
To evaluate the performance of our method, we use several metrics, including Mean Squared Error (MSE), Relative Error (RE), Peak Signal-to-Noise Ratio (PSNR), Structural Similarity Index (SSIM), Dice coefficient, and 95th percentile Hausdorff distance (HD95). Detailed definitions of these metrics are provided in Appendix~\ref{app:metric}.

\textbf{Computational Platform} Our methods were implemented in all experiments using the \textit{Python} programming language in combination with the \textit{PyTorch} framework. The experiments were conducted on a Windows 11 PC equipped with an Nvidia 4060 Ti GPU and a 13th Gen Intel(R) Core(TM) i7-13700KF 3.40 GHz CPU.

\textbf{Parameters} We set the weights of the composite loss function as follows if not specified: for fluid subproblem, $\alpha_\text{data} = 1.0$, $\alpha_\text{pde}=1.0\times 10^{-3}$, $\alpha_\text{bd}=1.0$, $\alpha_\text{in}=0.1$; for geometry subproblem, $\alpha_\text{reg}=1.0$, $\alpha_\text{bc}=1.0\times 10^{1}$, $\alpha_\text{lap}=1$.

\subsection{Flow Image over Converging Channel}

\begin{figure}
    \centering
    \includegraphics[width=\linewidth]{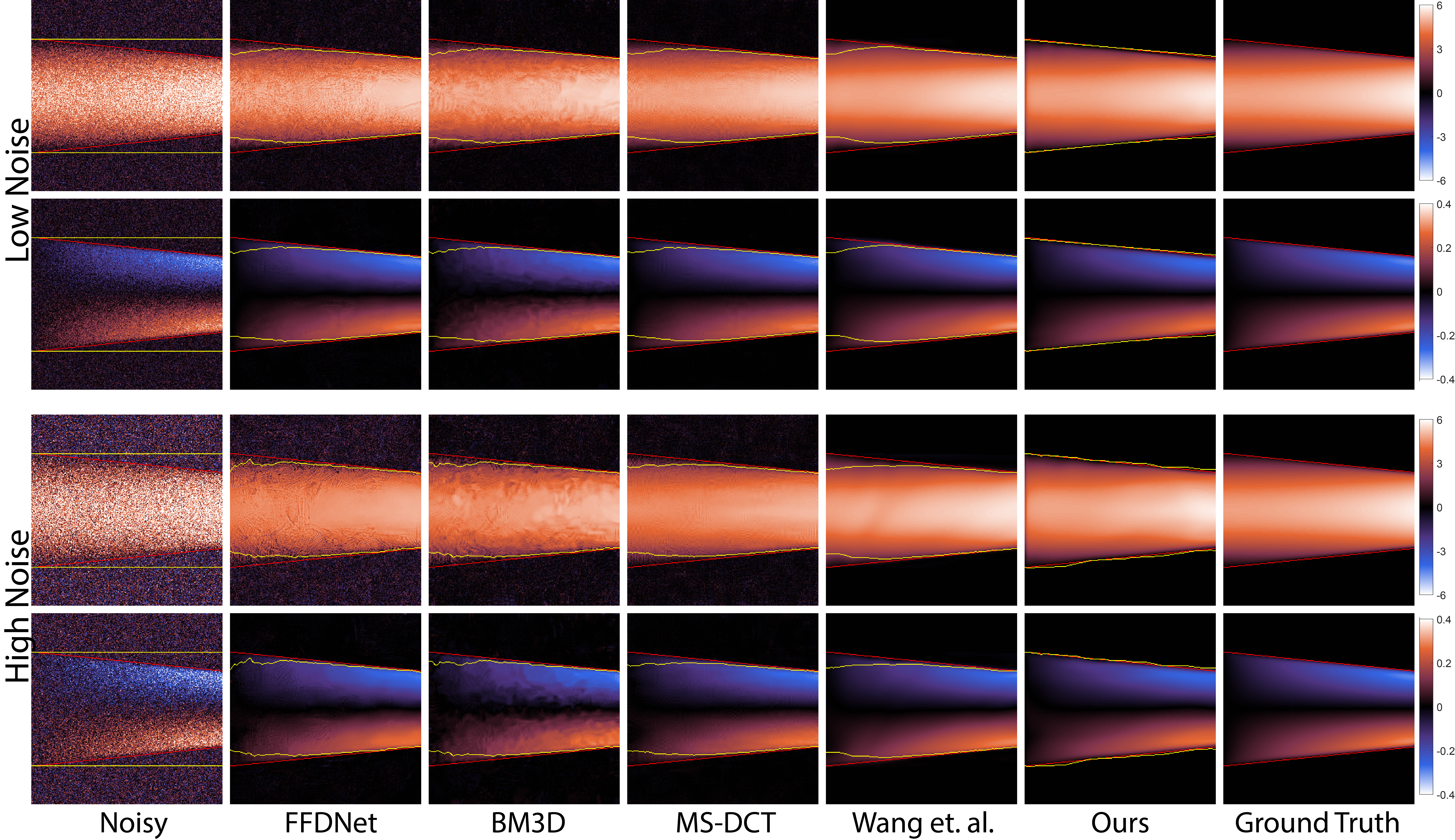}
    \caption{Qualitative comparison of flow and region reconstruction. Red contours show ground truth; yellow contours indicate predicted domains. Top two rows: low noise group; bottom two: high noise group. Each pair shows $x$- (top) and $y$-direction (bottom) velocity. Yellow contour in "Noisy" column marks the initialization.}
    \label{fig:exp_convCH}
\end{figure}
\begin{table}[!htp]
\centering
\begin{tabular}{c|c|cccc|ccc}
\hline
        &           &  \multicolumn{4}{c}{Flow Image Reconstruction}&  \multicolumn{3}{c}{Fluid Domain Segmentation}    \\\cline{3-9}
        &           &  RE       &  MSE      &  PSNR     &  SSIM     &  MSE      &  Dice     &  HD95  \\\hline
        &   Noise   & 0.4974    & 0.4742    & 23.119    & 0.1962    & 0.0961    & 0.9119    & 24.739\\
        &   FFDNet  & 0.0827    & 0.0973    & 32.935    & 0.7452    & 0.0454    & 0.9522    & 16.454\\
\multirow{2}{*}{Low}     
        &   BM3D    & 0.0348    & 0.0483    & 31.276    & 0.8870    & 0.0457    & 0.9518    & 16.342\\
        &   MC-DCT  & 0.0411    & \b{0.0383}& 34.523    & 0.9172    & 0.0476    & 0.9498    & 16.616\\
        & Wang \etal& 0.0648    & 0.0486    & 34.921    & \b{0.9855}& 0.0407    & 0.9573    & 16.000\\
        &   Ours    & \b{0.0124}& 0.0414    & \b{35.298}& 0.9847    & \b{0.0067}& \b{0.9933}& \b{4.000}\\
        \hline
        &   Noise   & 1.9621    & 1.7588    & 17.304    & 0.0645    & 0.0961    & 0.9119    & 24.739\\
        &   FFDNet  & 0.1656    & 0.3241    & 29.386    & 0.6564    & 0.0521    & 0.9358    & 19.417\\
\multirow{2}{*}{High}    
        &   BM3D    & 0.1279    & 0.1012    & 32.130    & 0.7733    & 0.0529    & 0.9350    & 20.105\\
        &   MC-DCT  & 0.0891    & 0.0727    & 31.283    & 0.8224    & 0.0681    & 0.9292    & 19.647\\
        & Wang \etal& 0.0577    & 0.0495    & 33.816    & 0.9637    & 0.0439    & 0.9539    & 21.000\\
        &   Ours    & \b{0.0150}& \b{0.0410}& \b{35.296}& \b{0.9841}& \b{0.0114}& \b{0.9887}& \b{4.000}\\
        \hline
\hline
\end{tabular}
\caption{Quantitative results for flow images in the converging channel geometry, comparing FFDNet \cite{zhang2018ffdnet}, BM3D \cite{lebrun2012analysis}, MC-DCT \cite{yu2011dct}, Wang \etal \cite{wang2022dense}, and the proposed method.}
\label{tb:result_convCH}
\end{table}

\begin{figure}
    \centering
    \includegraphics[width=\linewidth]{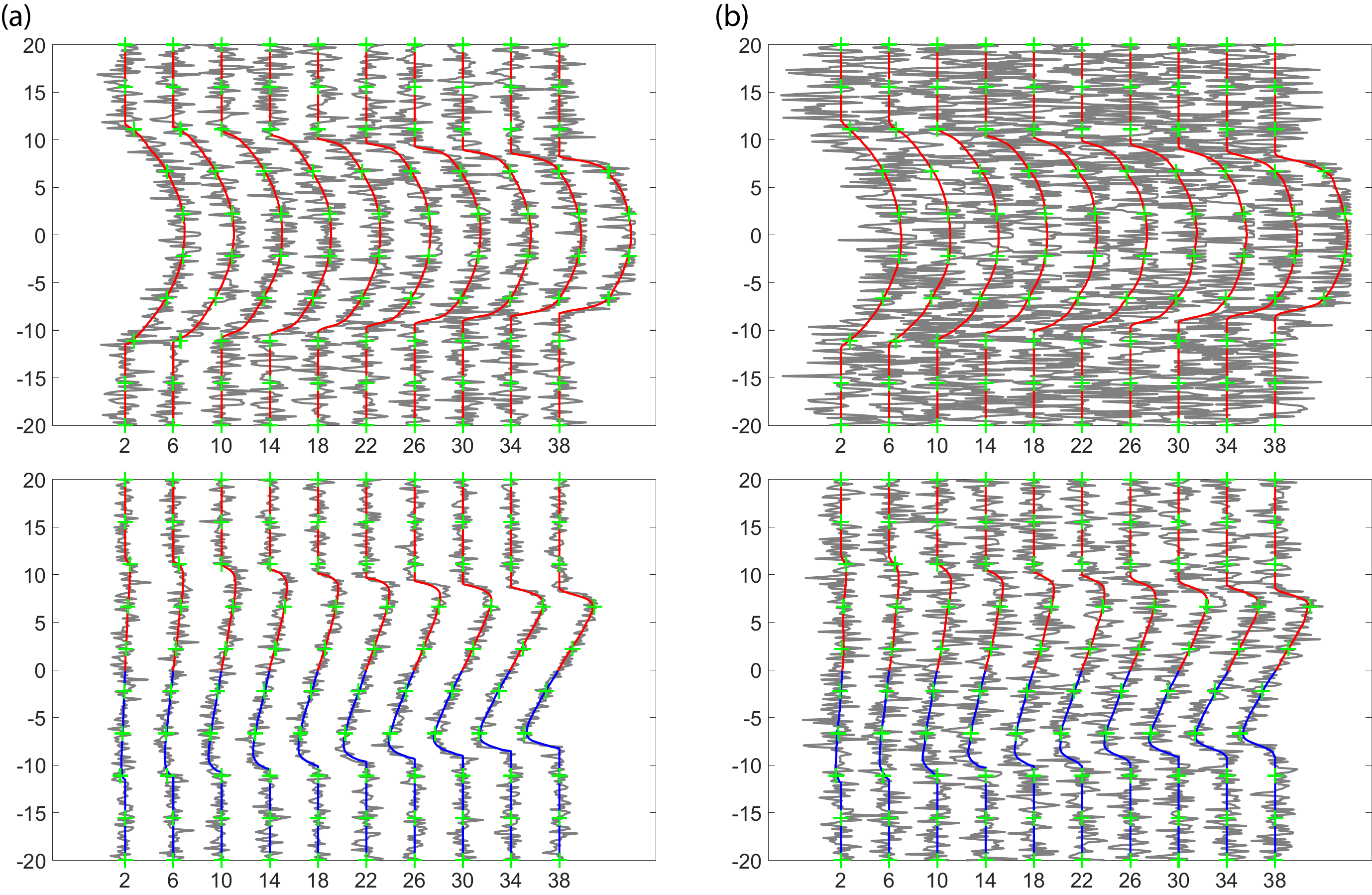}
    \caption{Discrepancies in velocity fields among the noisy input (gray), reconstructed results (red for positive, blue for negative), and ground truth (green “+”). The upper indicates the velocity in $x$-axis while the lower for $y$-axis. (a) Low noise; (b) High noise.}
    \label{fig:exp_convCH_profile}
\end{figure}

To validate the proposed method, we conducted experiments on synthetic flow images within a converging channel. The data consists of a two-channel image of size $256 \times 256$, representing the vertical and horizontal components of the velocity field over a physical domain measuring $40\ \mathrm{cm} \times 40\ \mathrm{cm}$. In the simulation setup, the inlet boundary velocity follows a parabolic profile with a maximum velocity of $5\ \mathrm{cm/s}$. We assume a Newtonian fluid with dynamic viscosity $\mu = 0.035\ \mathrm{g \cdot cm^{-1} \cdot s^{-1}}$ and density $\rho = 1.05\ \mathrm{g \cdot cm^{-3}}$, corresponding to a Reynolds number of $150$ based on a characteristic length of $5\ \mathrm{cm}$ and a characteristic velocity of $1\ \mathrm{cm/s}$. Either low or high levels of Gaussian noise were added to the synthesized images. For low noise, the $x$-direction flow had zero-mean noise with standard deviation 1, and the $y$-direction flow had zero-mean noise with standard deviation 0.01. For high noise, the $x$-direction flow had a standard deviation of 2, and the $y$-direction flow had a standard deviation of 0.02. The training was conducted for 240 iterations, alternating between 40 iterations of the fluid subproblem and 40 iterations of the geometry subproblem. Each fluid subproblem involves 50,000 collocation points. The entire training process takes approximately 5 minutes for the converging channel geometry case.

We compared our method with learning-based approaches like FFDNet \cite{zhang2018ffdnet} and traditional image processing techniques such as BM3D \cite{lebrun2012analysis}, multiscale DCT \cite{yu2011dct}, and a physics-based method \cite{wang2022dense}. The evaluation included both quantitative metrics and qualitative visualizations for flow image reconstruction and region segmentation. Since competing methods do not provide segmentations directly, we applied the geodesic active contour model \cite{caselles1997geodesic} to extract flow regions from their denoised outputs.

Figure~\ref{fig:exp_convCH} shows qualitative comparisons. The red contours indicate the ground-truth flow regions, and the yellow contours show those inferred by each method. The first two rows correspond to low-noise inputs, and the last two correspond to high-noise conditions. For each noise level, the velocity components in the $x$- and $y$-directions are shown in the upper and lower rows, respectively. The yellow contour in the "Noisy" column reflects the initial domain estimate used to initialize the applicable methods.

As shown in Figure~\ref{fig:exp_convCH} and Table~\ref{tb:result_convCH}, our method consistently outperforms the others. Regarding flow image quality, conventional image denoising models are unable to fully eliminate artifacts, often resulting in unnatural or physically implausible flow patterns. While the physics-based approach by Wang \etal~\cite{wang2022dense} achieves improved performance, its accuracy remains limited in both mechanical reconstruction and domain segmentation, which arises from the lack of coupled geometry evolution and the absence of domain boundary information during the optimization process. In contrast, by incorporating the physics of fluid flow through the Navier–Stokes equations and the associated domain evolution, our approach produces more realistic and accurate reconstructions in both fluid mechanics and domain. The PINN model with boundary information used for velocity estimation further ensures robust and efficient optimization. Regarding region segmentation, it is observed that the active contour model, when applied to denoised images, often yields inaccurate boundaries, particularly near wall regions and close to the inlet, due to its reliance solely on image gradients. In contrast, our method benefits from the underlying velocity field information through a registration-based manner, enabling it to recover flow regions that closely match the ground truth and extract a more accurate flow region.

\begin{figure}
    \centering
    \includegraphics[width=0.7\linewidth]{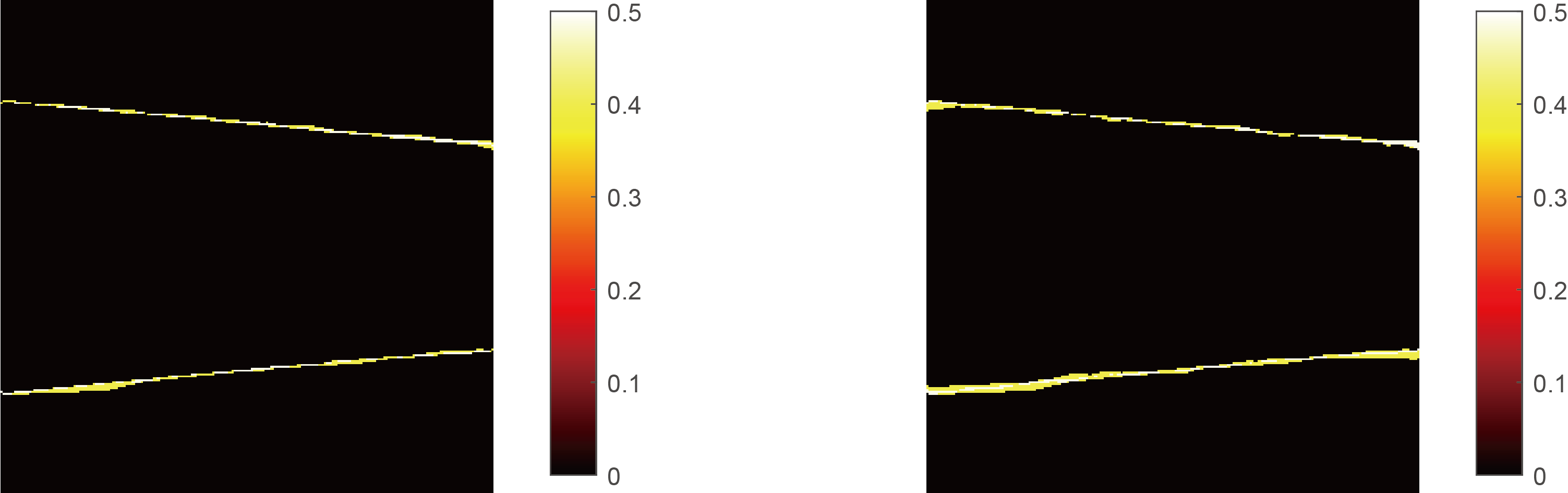}
    \caption{Illustration of segmentation uncertainty for the converging channel geometry with Gaussian noise. The maps show the standard deviation computed from five noisy images with randomly generated noise. The left corresponds to the low-noise case, while the right represents the high-noise case.}
    \label{fig:convCH_UQ}
\end{figure}

To closely examine the noisy measurements, the reconstructed velocity field, and the ground truth, we present Figure 6. The upper row shows the $x$-velocity component, and the bottom row shows the $y$-velocity component. In each plot, the curves represent velocity magnitudes sampled along a horizontal line at a fixed $x$-location. The horizontal displacement from the vertical indicates speed — larger deviations correspond to higher velocities. Reconstructed velocities are colored red for positive (rightward/upward) and blue for negative (leftward/downward) values; gray curves represent noisy observations, and green “+” markers indicate ground truth. From the profile observations, the reconstructed velocities closely match the ground truth across inflow, outflow, and intermediate regions. They accurately capture the fluid dynamics and effectively suppress noise, as indicated by the gray curves, while preserving physical consistency.

We also evaluate the uncertainty of our model under different levels of Gaussian noise. Specifically, we apply our method to five independently generated noisy images at both low and high noise levels, and compute the standard deviation of the resulting segmented regions. The corresponding uncertainty map, shown in Figure~\ref{fig:convCH_UQ}, indicates that the segmentation uncertainty remains relatively low even under high noise. This robustness can be attributed to the global guidance provided by the registration-based geometry modeling.

\subsection{Flow Image over Aorta Vessel}

\begin{figure}
    \centering
    \includegraphics[width=\linewidth]{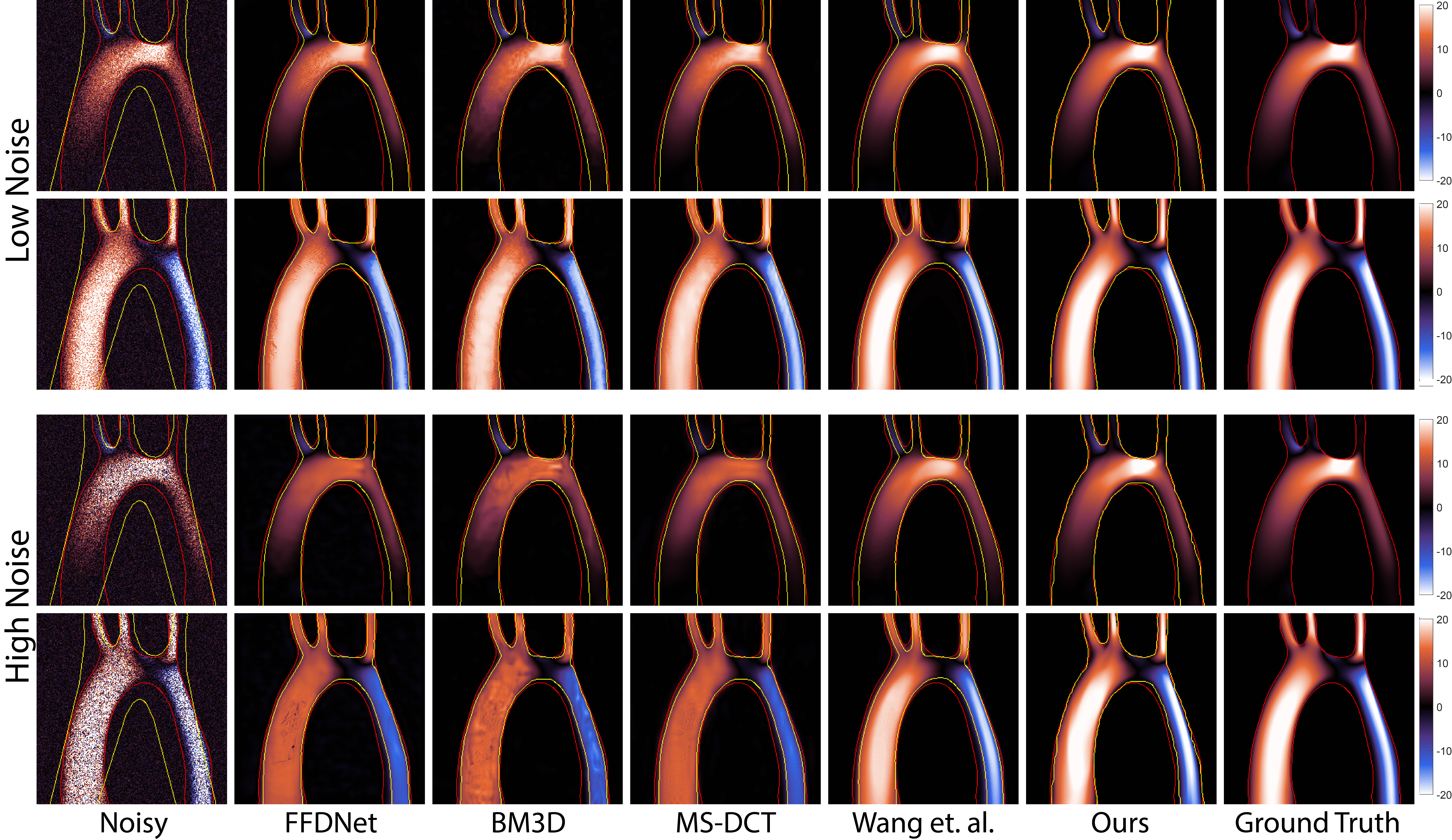}
    \caption{Qualitative comparison of flow and region reconstruction. Red contours show ground truth; yellow contours indicate predicted domains. Top two rows: low noise group; bottom two: high noise group. Each pair shows $x$- (top) and $y$-direction (bottom) velocity. Yellow contour in "Noisy" column marks the initialization.}
    \label{fig:exp_aorta}
\end{figure}

\begin{table}[!htp]
\centering
\begin{tabular}{c|c|cccc|ccc}
\hline
        &           &  \multicolumn{4}{c}{Flow Image Reconstruction}&  \multicolumn{3}{c}{Fluid Domain Segmentation}    \\\cline{3-9}
        &           &  RE       &  MSE      &  PSNR     &  SSIM     &  MSE      &  Dice     &  HD95  \\\hline
        &   Noise   & 2.1716    & 5.8275    & 24.375    & 0.2967    & 0.1627    & 0.7904    & 26.476\\
        &   FFDNet  & 0.1456    & 0.8605    & 32.514    & 0.9251    & 0.0517    & 0.9194    & 15.454\\
\multirow{2}{*}{Low}     
        &   BM3D    & 0.1371    & 0.7917    & 32.276    & 0.9176    & 0.0529    & 0.9119    & 14.342\\
        &   MC-DCT  & 0.1670    & 0.7214    & 34.675    & 0.9353    & 0.0613    & 0.8922    & 14.000\\
        & Wang \etal& 0.1651    & 0.4825    & 35.642    & 0.9788    & 0.0587    & 0.9024    & 13.928\\
        &   Ours    & \b{0.0847}& \b{0.4133}& \b{38.447}& \b{0.9869}& \b{0.0281}& \b{0.9574}& \b{8.000}\\
        \hline
        &   Noise   & 3.1875    & 24.0538   & 18.680    & 0.2487    & 0.1627    & 0.7904    & 26.476\\
        &   FFDNet  & 0.1566    & 3.4682    & 26.672    & 0.8915    & 0.0637    & 0.9019    & 16.170\\
\multirow{2}{*}{High}     
        &   BM3D    & 0.2654    & 3.7285    & 24.287    & 0.8835    & 0.0607    & 0.8881    & 16.854\\
        &   MC-DCT  & 0.2753    & 4.1753    & 27.176    & 0.9117    & 0.0648    & 0.8767    & 15.000\\
        & Wang \etal& 0.2328    & 0.8734    & 33.374    & 0.9667    & 0.0761    & 0.8937    & 15.000\\
        &   Ours    & \b{0.0964}& \b{0.5973}& \b{36.136}& \b{0.9804}& \b{0.0296}& \b{0.9540}& \b{8.000}\\
        \hline
\hline
\end{tabular}
\caption{Quantitative results for flow images in the aorta geometry, comparing FFDNet \cite{zhang2018ffdnet}, BM3D \cite{lebrun2012analysis}, MC-DCT \cite{yu2011dct}, Wang \etal \cite{wang2022dense}, and the proposed method.}
\label{tb:result_Aorta}
\end{table}

\begin{figure}
    \centering
    \includegraphics[width=\linewidth]{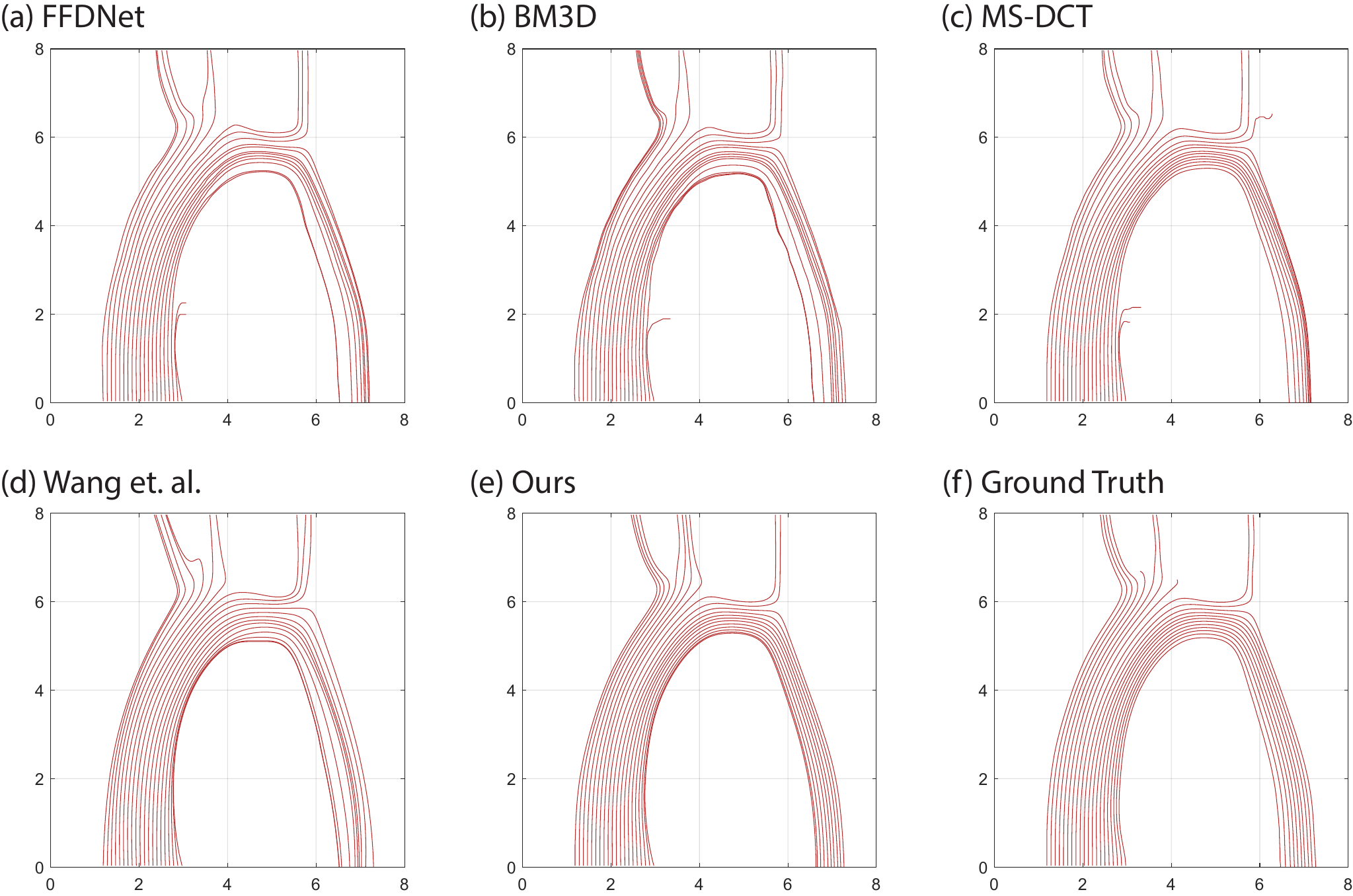}
    \caption{Streamline visualization reconstructed from a low-noise flow image under identical settings.}
    \label{fig:stream_low}
\end{figure}
\begin{figure}
    \centering
    \includegraphics[width=\linewidth]{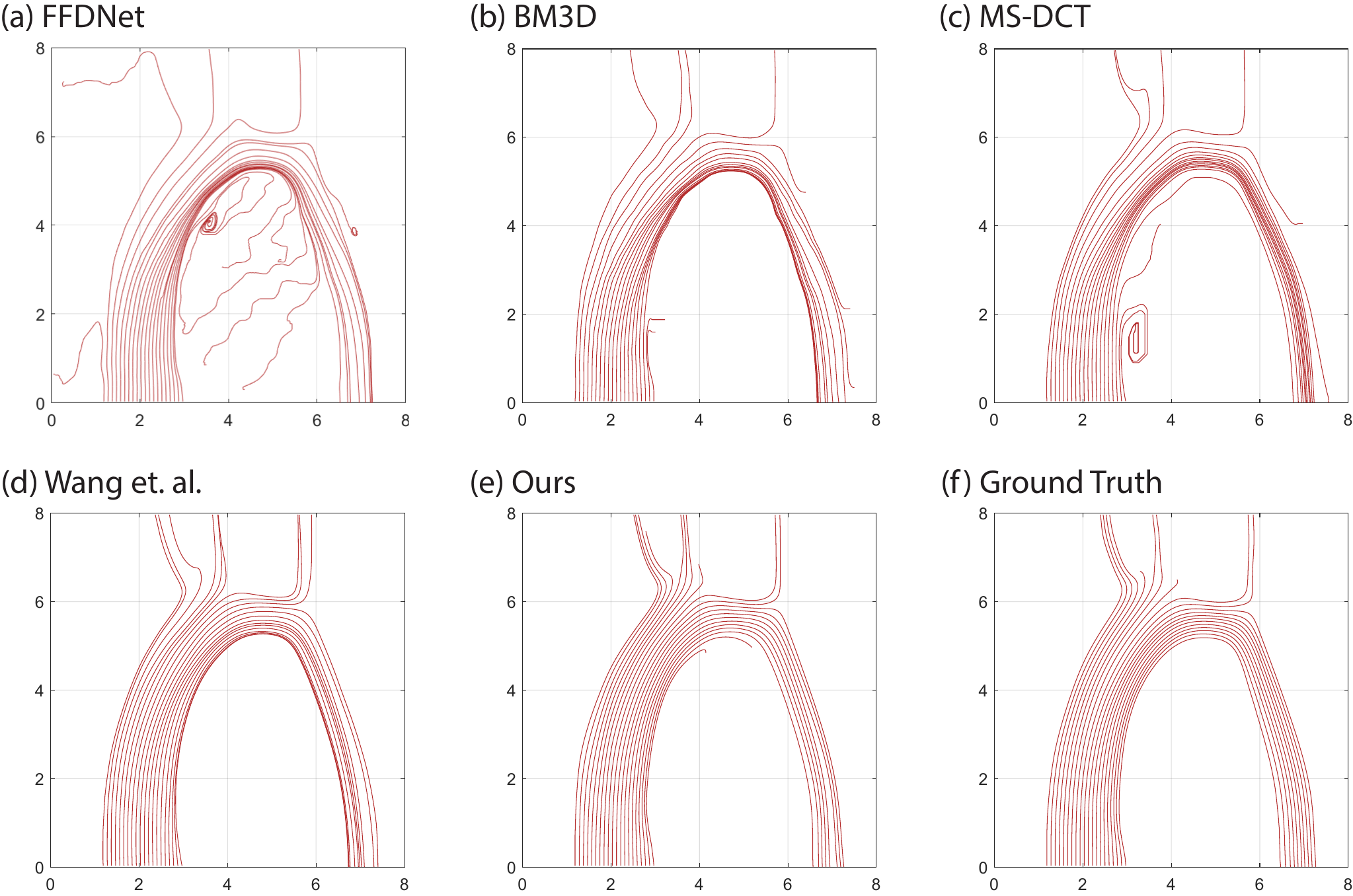}
    \caption{Streamline visualization reconstructed from a high-noise flow image under identical settings.}
    \label{fig:stream_high}
\end{figure}

To evaluate our method in a more realistic setting, we conducted experiments on synthetic flow images within an aorta geometry. The $256 \times 256$ images represent an $8\ \mathrm{cm} \times 8\ \mathrm{cm}$ domain, with the radius of the aorta around $2\ \mathrm{cm}$. Following the simulation setup in \cite{updegrove2017simvascular,simvascular}, a parabolic inlet velocity profile with a peak of $20\ \mathrm{cm/s}$ was used. The dynamic viscosity is set to $\mu = 0.035\ \mathrm{g \cdot cm^{-1} \cdot s^{-1}}$ and the density to $\rho = 1.05\ \mathrm{g \cdot cm^{-3}}$. The Reynolds number for the simulation is $600$, based on a characteristic length of $1\ \mathrm{cm}$ and a characteristic velocity of $20\ \mathrm{cm/s}$. Signal-dependent noise was added following the model in \eqref{eq:noise_model}. For low noise, the standard deviations were set to $\sigma_1 = 1$ and $\sigma_2 = 2$ for the two directions, while for high noise, both directions used $\sigma_1 = 2$ and $\sigma_2 = 2$. The training was performed for 400 iterations, alternating between 40 iterations of the fluid subproblem and 40 iterations of the geometry subproblem. Each fluid subproblem uses 50,000 collocation points. The entire training process takes approximately 9 minutes for the aorta geometry case.

As in the previous case, we compared our method with FFDNet \cite{zhang2018ffdnet}, BM3D \cite{lebrun2012analysis}, multiscale DCT \cite{yu2011dct}, and a physics-based method \cite{wang2022dense}, evaluating both flow reconstruction and segmentation. For methods without native segmentation, the geodesic active contour model \cite{caselles1997geodesic} was applied. Figure~\ref{fig:exp_aorta} presents qualitative results, following the same layout as Figure~\ref{fig:exp_convCH}, with red contours indicating ground truth and yellow contours showing inferred domains under both low and high noise conditions.

As shown in the results in Figure~\ref{fig:exp_aorta} and Table~\ref{tb:result_Aorta}, our method consistently outperforms the others in all evaluation metrics. Conventional denoising methods such as FFDNet, BM3D, and multiscale DCT struggle to effectively suppress noise, particularly under high-noise conditions. This limitation is especially evident in the last two columns of the figure, where these methods produce flow fields with pronounced artifacts or physically implausible patterns. In such cases, the reconstructed images fail to preserve the underlying flow structure and offer limited utility for meaningful fluid dynamics analysis. While the physics-based approach by Wang \etal~\cite{wang2022dense} achieves improved quantitative performance compared to traditional denoising methods, its accuracy remains insufficient, primarily due to the absence of coupled geometry evolution and of the associated domain boundary information during the optimization process.

In contrast, our method consistently produces clean and physically realistic reconstructions, as demonstrated in Figures \ref{fig:stream_low} and \ref{fig:stream_high}, which show reconstructed streamlines from low- and high-noise images, respectively. This improvement is primarily attributed to the incorporation of physical constraints via the Navier–Stokes equations and the robust optimization enabled by the differentiable PINN model, as well as the coupled domain evolution during the process. By enforcing consistency with the governing physics of fluid flow and the domain, our differentiable approach effectively suppresses noise while preserving essential flow features embedded in noisy measurements.

Regarding flow region segmentation, intensity-based methods, such as geodesic active contours, often result in inaccurate boundary detection, particularly near vessel walls, where velocity naturally decreases. This misidentification leads to segmentations that deviate from the true physiological flow region.

Our method, on the contrary, performs flow reconstruction and region segmentation in a coupled manner, allowing each task to inform and refine the other. This mutual reinforcement improves the accuracy of both the reconstructed flow fields and the segmented domains. The synergy between physics-informed reconstruction and segmentation contributes to the overall performance advantage of our approach, as confirmed by both visual results and quantitative evaluations. For the high‐noise aorta case, the segmented domain occasionally shows jagged boundaries. This was addressed in \cite{kontogiannis2022joint} by adding a viscous term that implies motion of the domain's boundary by its mean curvature \cite{merriman1992diffusion}. Instead, we apply morphological opening followed by closing with a disk of radius $r=6$, which effectively bounds the absolute value of the boundary curvature by $1/r$.

\begin{figure}
    \centering
    \includegraphics[width=0.7\linewidth]{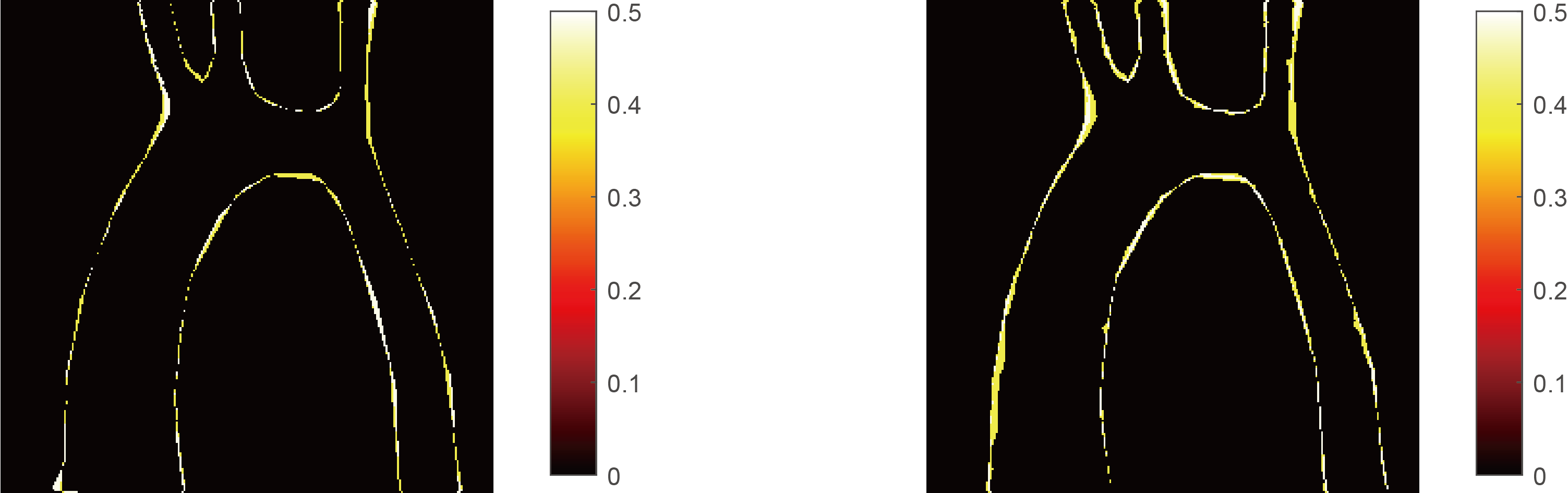}
    \caption{Illustration of segmentation uncertainty for the aorta geometry with signal-dependent noise. The maps show the standard deviation computed from five noisy images with randomly generated noise. The left corresponds to the low-noise case, while the right represents the high-noise case.}
    \label{fig:Aorta_UQ}
\end{figure}

To assess the segmentation uncertainty in the aorta geometry case with signal-dependent noise, we follow a procedure similar to the converging channel experiment. Our method is applied to five independently generated noisy images at both low and high noise levels, and the standard deviation of the resulting segmented regions is calculated. The uncertainty map in Figure~\ref{fig:Aorta_UQ} shows that the segmentation remains stable, with controllable variations even under high noise conditions. This consistency highlights the robustness of the proposed framework, benefiting from the globally guided deformation in the registration-based geometry modeling.

\begin{figure}
    \centering
    \includegraphics[width=\linewidth]{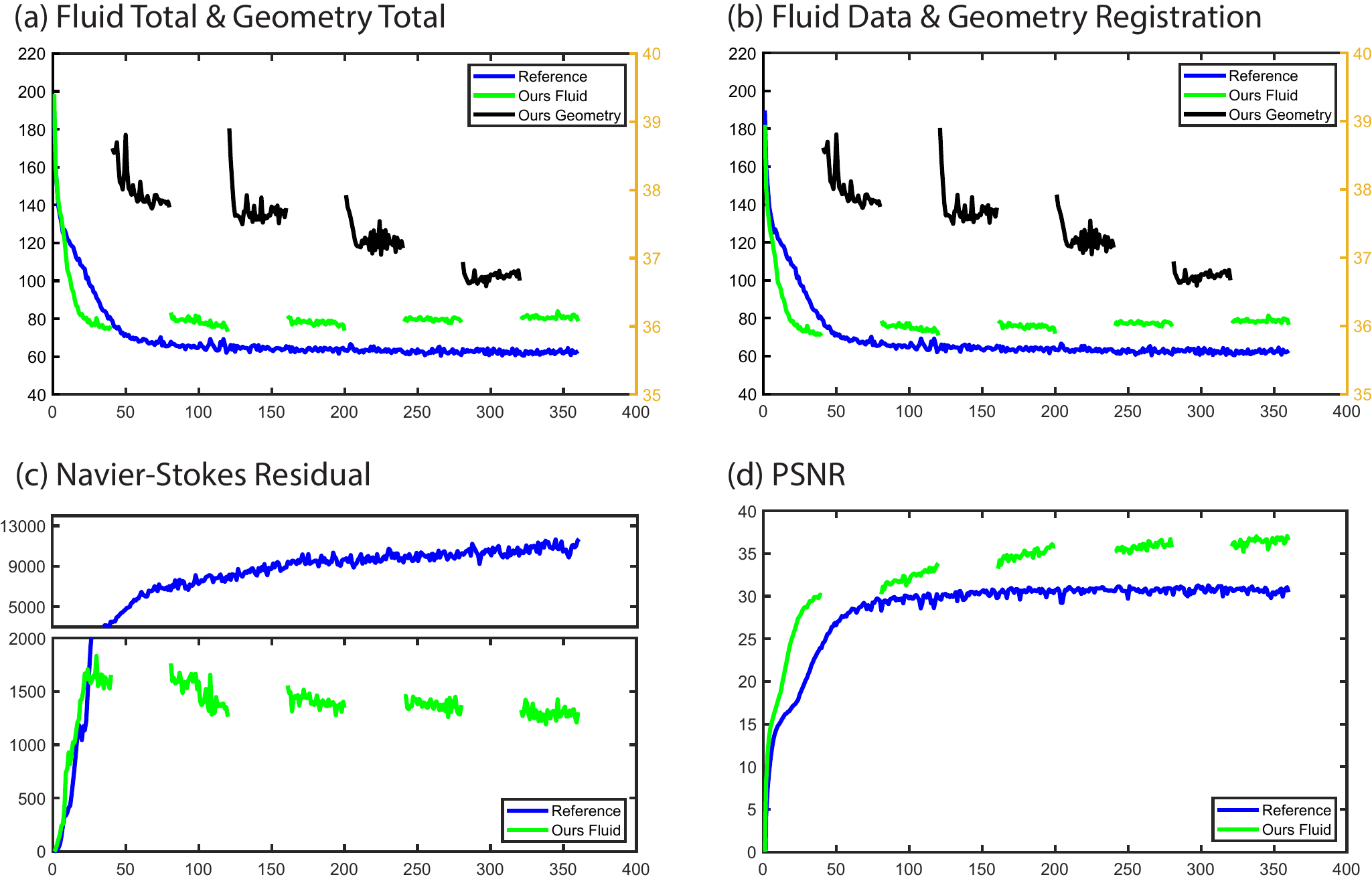}
    \caption{Training dynamics and convergence behavior for the aorta geometry case by our method and the reference. The reference method is done by removing the Navier–Stokes residual terms and boundary conditions, training the model solely with the data-fidelity term. The green curve represents the total loss of the fluid subproblem, the black curve corresponds to the geometry subproblem, and the blue curve denotes the result obtained by the reference model trained with only the data-fidelity term. (a) Total fluid and geometry losses; (b) data loss in the fluid subproblem and registration loss in the geometry subproblem; (c) Navier–Stokes residuals; (d) PSNR evolution during training.}
    \label{fig:aorta_loss}
\end{figure}

We further examine the error plots in Figure~\ref{fig:aorta_loss} to better understand the training dynamics and the role of the physics-informed component, by comparing our method with a reference approach. The reference method is obtained by removing the Navier–Stokes residual terms and boundary conditions, and training the model solely using the data-fidelity term. In these plots, for our method, the green curve represents the total loss of the fluid subproblem, while the black curve corresponds to the geometry subproblem. The blue curve shows the result obtained using the reference approach. As illustrated in Figure~\ref{fig:aorta_loss}(a), the alternating behavior between the total fluid loss (y-axis on the left) and the total geometry loss (y-axis on the right) clearly demonstrates the effect of our alternating optimization strategy. Because our method incorporates physical constraints, the data error decreases more slowly than in the reference model. However, as shown in Figure~\ref{fig:aorta_loss}(c), the minimization of the Navier–Stokes residual leads to significantly improved reconstruction accuracy, which is reflected by the better PSNR in Figure~\ref{fig:aorta_loss}(d).

\begin{figure}
    \centering
    \includegraphics[width=0.4\linewidth]{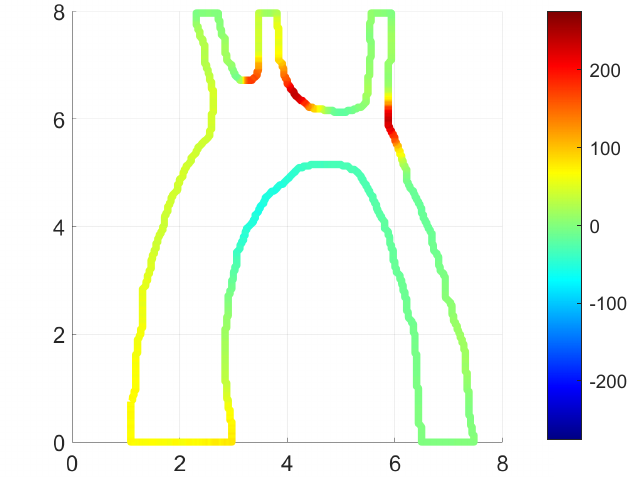}
    \caption{Wall shear stress (WSS) distribution on the aortic wall obtained from the reconstructed solution in the high-noise case. The color codes the spatial variation of shear stress magnitude ($\mathrm{dyn / cm^{2}}$).}
    \label{fig:WSS}
\end{figure}

Wall shear stress (WSS) is a clinically significant hemodynamic quantity. Therefore, we further estimate this variable and visualize its spatial distribution in Figure \ref{fig:WSS}. The results are obtained from the high-noise aorta geometry case. Using the reconstructed velocity networks, which allow direct computation of velocity gradients, together with the pressure network, we evaluate the WSS distribution on the vessel boundary. A color-coded plot with an accompanying color bar is provided to illustrate its spatial variation. Combined with the velocity visualization shown in Figure \ref{fig:exp_aorta}, we observe that elevated shear stress occurs in regions where the flow direction changes markedly due to geometric constraints of the vessel, while lower WSS values appear in areas where the velocity field propagates more smoothly.

\subsection{Self Ablation}

Here, we perform ablation studies to assess the impact of different weighting parameters in the optimization of the fluid and geometry subproblems. The high-niose flow image with Aorta geometry is used.

\textbf{Test on super-resolution}
Our method is built upon a meshless PINN solver constrained by the Navier–Stokes equations. Thanks to the meshless formulation, the reconstructed velocity field can be evaluated at arbitrary spatial resolution while maintaining solution accuracy. To demonstrate this advantage, we include a super-resolution experiment using aorta flow data of sizes $64 \times 64$, $128 \times 128$, and the originally $256 \times 256$, For each case, we apply the same reconstruction algorithm, and then evaluate the resulting velocity field on a $256 \times 256$ grid. As shown in Table~\ref{tb:self}, all reconstructed high-resolution velocity fields closely match the $256 \times 256$ ground truth, as reflected by the quantitative metrics. This experiment confirms that our method can reliably recover high-quality flow fields even from low-resolution measurements.

\textbf{Test on $\alpha_\text{data}$ : }
We conducted an ablation study to evaluate the effect of the data fitting weight parameter $\alpha_\text{data}$, which plays a critical role in aligning the predicted velocity field with the observed data. The best overall performance was achieved at $\alpha_\text{data} = 10^0$, striking a strong yet balanced emphasis on data fidelity while maintaining physical regularization. At $\alpha_\text{data} = 10^1$, the model still performed well, although signs of overfitting appeared as the network prioritized fitting noisy data, which reduced generalization. Lowering the weight to $10^{-1}$ caused a slight performance drop due to weaker data alignment, particularly in regions with fine flow features. When $\alpha_\text{data}$ was reduced further to $10^{-2}$, the performance suffered significantly: the data constraint was too weak to effectively guide the model, causing the optimization to rely primarily on regularization and resulting in poor recovery of the flow field and domain structure.

\begin{figure}[t]
    \centering
    \includegraphics[height=5cm]{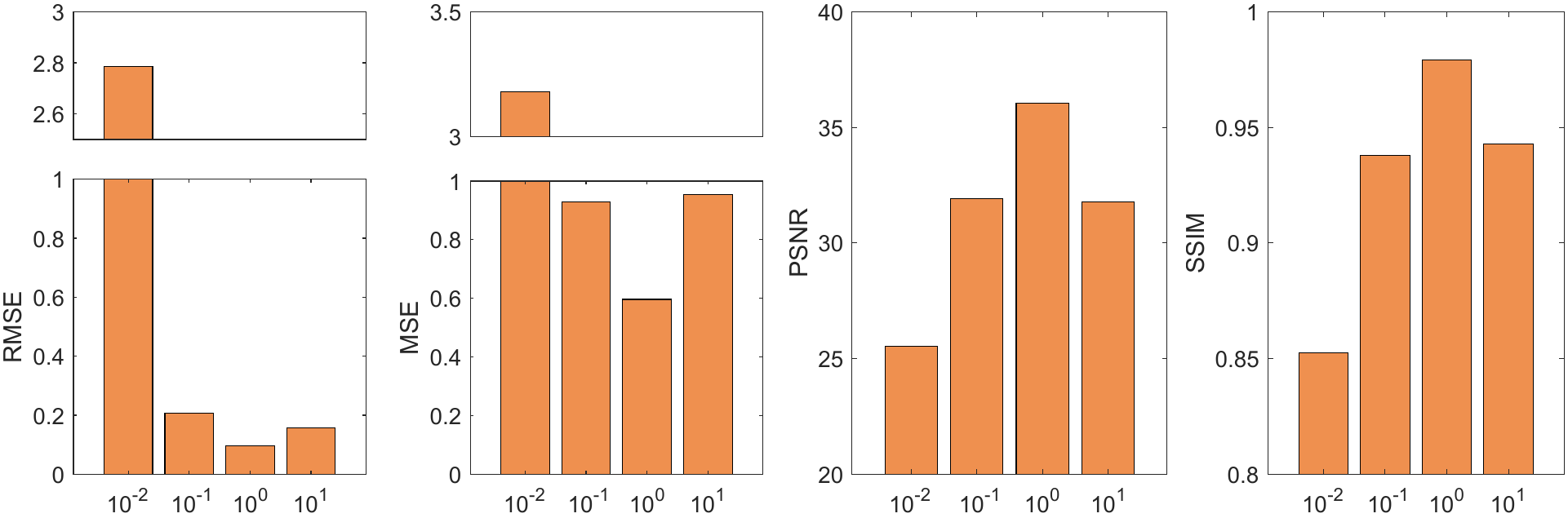}
    \caption{Ablation study evaluating the impact of the data fitting weight $\alpha_\text{data}$ on training performance.}
    \label{fig:ablation_data}
\end{figure}

\textbf{Test on $\alpha_\text{bc}$ : }
We conducted a self-ablation study on the sensitivity to the Beltrami coefficient regularization weight, $\alpha_\text{bc}$. The best performance occurred at $\alpha_\text{bc} = 10^1$, balancing mapping regularization and registration. At $\alpha_\text{bc} = 10^0$, results were acceptable, although minor deterioration appeared, probably due to increased deformation causing localized shrinkage. Lowering to $\alpha_\text{bc} = 10^{-1}$ significantly degraded performance, as weak regularization allowed excessive domain deviation, topological errors, and outlier regions that harmed precision. Conversely, at $\alpha_\text{bc} = 10^2$, over-regularization constrained the model, limiting domain propagation and adaptation to subtle variations, leading to incomplete or inaccurate reconstructions. This over-constraint effect intensified at $\alpha_\text{bc} = 10^3$.
\newcolumntype{D}{>{\centering\arraybackslash}p{0.12\columnwidth}}
\newcommand{\owc}[1]%
    {%
        \global\let\oldcol\currowcolor%
        \global\def\currowcolor{\oldcol!50!#1}%
    }%
\begin{table}[!htp]
\centering
\begin{tabu}{c|DDDD}
\hline\hline
Resol.              &  RE       &  MSE      &  PSNR     &  SSIM \\\hline
$64^2$              & 0.1381    &  0.6891   & 33.648    & 0.9498\\
$128^2$             & 0.1207    &  0.6343   & 35.187    & 0.9679\\
$256^2$             & 0.0969    &  0.5966   & 36.053    & 0.9792\\\hline
\hline
$\alpha_\text{data}$    &  RE       &  MSE      &  PSNR     &  SSIM \\\hline
$10^{-2}$               & 2.7861    &  3.1784   & 25.542    & 0.8524\\
$10^{-1}$               & 0.2071    &  0.9284   & 31.917    & 0.9378\\
$10^{ 0}$               & 0.0969    &  0.5966   & 36.053    & 0.9792\\  
$10^{ 1}$               & 0.1578    &  0.9546   & 31.786    & 0.9427\\\hline
\hline
$\alpha_\text{bc}$  & $\|\mu\|_2$ & Dice & HD & $||\mu||_\infty<1$  \\\hline
$10^{-1}$           &  1.078        & 0.8817    & 15.757    & \ding{55}\\
$10^{ 0}$           &  0.337        & 0.9375    & 10.757    & \ding{51}\\
$10^{ 1}$           &  0.068        & 0.9542    &  7.616    & \ding{51}\\
$10^{ 2}$           &  0.005        & 0.9056    & 13.572    & \ding{51}\\
$10^{ 3}$           &  0.000        & 0.8271    & 19.278    & \ding{51}\\\hline
\hline
\end{tabu}
\caption{Ablation study on the impact of flow image resolution, varying data fidelity weight $\alpha_\text{data}$ and the Beltrami regularization weight $\alpha_\text{bc}$ for flow images in an aorta-shaped domain under high-noise conditions.}
\label{tb:self}
\end{table}

\begin{figure}[t]
    \centering
    \includegraphics[height=5cm]{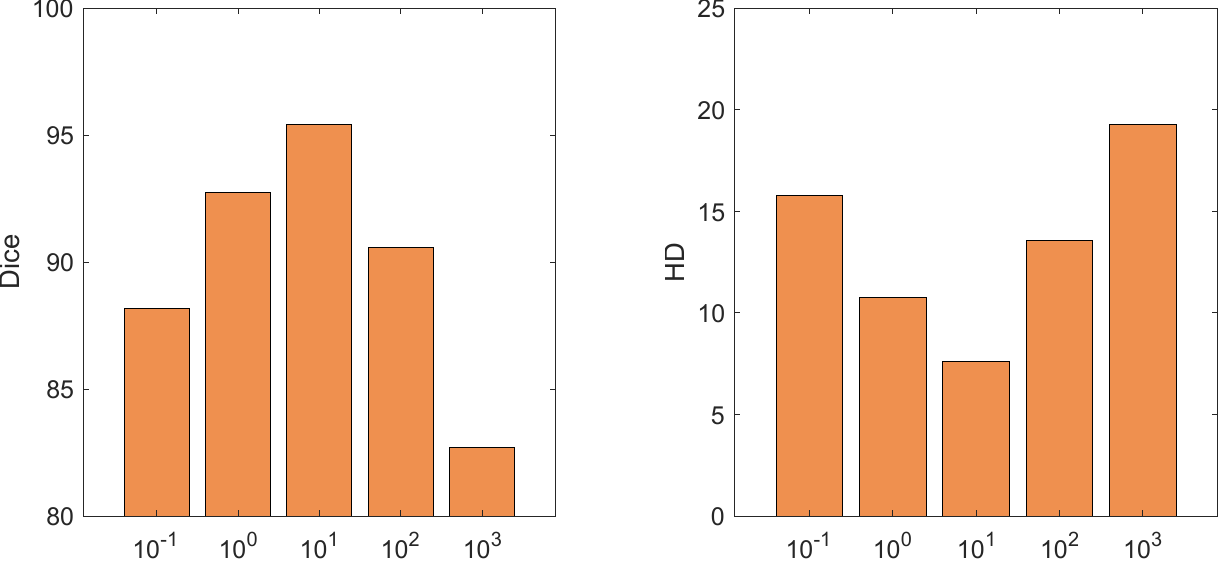}
    \caption{Ablation study evaluating the impact of Beltrami regularization weight $\alpha_\text{data}$ on training performance.}
    \label{fig:ablation_bc}
\end{figure}
\section{Conclusion}

In this work, we propose a novel framework that integrates deep learning with physical modeling to improve the reconstruction of 2D flow images under noisy measurements. The problem is formulated to enforce physical consistency through the Navier–Stokes equations and preserve anatomical structures using quasi-conformal mappings. It is decomposed into two subproblems: a fluid subproblem, which estimates the velocity field within a given domain using Navier–Stokes constraints, and a geometry subproblem, which refines the domain by aligning the reconstructed velocity with observed data via a quasi-conformal mapping. These subproblems are solved iteratively in a Gauss–Seidel fashion until convergence, providing a solution supported by theoretical guarantees. In practice, we employ a Physics-Informed Neural Network for the fluid subproblem and a UNet-based Mapping Estimator for the geometry subproblem.

We extensively validated our method through experiments. Initial tests on synthetic flow images in a converging channel under varying Gaussian noise levels validate the proposed method. Further evaluations on anatomically realistic aorta geometries with signal-dependent noise confirmed its effectiveness in complex real-world settings. Our method consistently outperformed the baseline approaches, delivering more accurate and higher-quality reconstructions. Ablation studies also highlighted the impact of weighting schemes and guided the selection of effective hyperparameters. 

Future work will extend this framework to 3D flow reconstruction and 4D dynamic flow analysis by replacing the 2D quasi-conformal mappings with Jacobian-based constraints or hyperelastic energy models to ensure smooth and bijective deformations in higher dimensions. We also plan to adapt the framework for higher Reynolds number flows by employing more expressive architectures, such as PirateNets \cite{wang2024piratenets}. Additionally, we aim to tackle more challenging imaging modalities, such as ultrasound Doppler imaging, which involve higher noise levels and real-time computational constraints. Furthermore, since the current Gauss–Seidel alternating scheme lacks a formal convergence guarantee, and the validation of the split problems relies primarily on empirical observations, we plan to develop a theoretical convergence analysis by adding additional regularization.

\bibliographystyle{siamplain}
\bibliography{references}

@string{cm = {Computational Mechanics}}

@string{cmame={Computer Methods in Applied Mechanics and Engineering}}

@string{iclr = {Proc. of IEEE Int. Conf. on Learning Representation}}

@string{tbe = {IEEE Trans. on Biomedical Engineering}}

@string{jsc = {Journal of Scientific Computing}}

@article{zhang2024full,
  title={Full 3D blood flow simulation in curved deformable vessels using physics-informed neural networks},
  author={Zhang, Han and Tai, Xue-Cheng},
  journal={Acta Mathematica Universitatis Comenianae},
  volume={93},
  number={4},
  pages={235--250},
  year={2024},
  publisher={Univerzita Komenskeho}
}

@article{updegrove2017simvascular,
  title={SimVascular: an open source pipeline for cardiovascular simulation},
  author={Updegrove, Adam and Wilson, Nathan M and Merkow, Jameson and Lan, Hongzhi and Marsden, Alison L and Shadden, Shawn C},
  journal={Annals of Biomedical Engineering},
  volume={45},
  pages={525--541},
  year={2017},
  publisher={Springer}
}

@misc{simvascular,
  author={Freidoonimehr, Navid and Arjomandi, Maziar and Zander, Anthony and Chin, Rey},
  title = {{SimVascular} Clinical Test Cases : Aortofemoral Normal - 2},
  note  = {https://simvascular.github.io/clinical/aortofemoral2.html},
  year={2025},
}

@article{zhang2021topology,
  title={Topology-preserving 3D image segmentation based on hyperelastic regularization},
  author={Zhang, Daoping and Lui, Lok Ming},
  journal={Journal of Scientific Computing},
  volume={87},
  number={3},
  pages={74},
  year={2021},
  publisher={Springer}
}

@article{crosetto2011parallel,
  title={Parallel Algorithms for Fluid-Structure Interaction Problems in Haemodynamics},
  author={Crosetto, Paolo and Deparis, Simone and Fourestey, Gilles and Quarteroni, Alfio},
  journal={SIAM Journal on Scientific Computing},
  volume={33},
  number={4},
  pages={1598--1622},
  year={2011},
  publisher={SIAM}
}

@article{quarteroni2004mathematical,
  title={Mathematical Modelling and Numerical Simulation of The Cardiovascular System},
  author={Quarteroni, Alfio and Formaggia, Luca},
  journal={Handbook of numerical analysis},
  volume={12},
  pages={3--127},
  year={2004},
  publisher={Elsevier}
}

@article{yan2022impact,
  title={Impact of Pressure Wire on Fractional Flow Reserve and Hemodynamics of The Coronary Arteries: A Computational and Clinical Study},
  author={Yan, Zhengzheng and Yao, Zhifeng and Guo, Weifeng and Shang, Dandan and Chen, Rongliang and Liu, Jia and Cai, Xiao-Chuan and Ge, Junbo},
  journal=tbe,
  volume={70},
  number={5},
  pages={1683--1691},
  year={2022},
  publisher={IEEE}
}

@article{bazilevs2008isogeometric,
  title={Isogeometric Fluid-Structure Interaction: Theory, Algorithms, and Computations},
  author={Bazilevs, Yuri and Calo, Victor M and Hughes, Thomas JR and Zhang, Yongjie},
  journal=cm,
  volume={43},
  pages={3--37},
  year={2008},
  publisher={Springer}
}

@article{figueroa2006coupled,
  title={A Coupled Momentum Method for Modeling Blood Flow in Three-Dimensional Deformable Arteries},
  author={Figueroa, C Alberto and Vignon-Clementel, Irene E and Jansen, Kenneth E and Hughes, Thomas JR and Taylor, Charles A},
  journal=cmame,
  volume={195},
  number={41-43},
  pages={5685--5706},
  year={2006},
  publisher={Elsevier}
}

@article{pijls1996measurement,
  title={Measurement of Fractional Flow Reserve To Assess The Functional Severity of Coronary-Artery Stenoses},
  author={Pijls, Nico HJ and de Bruyne, Bernard and Peels, Kathinka and van der Voort, Pepijn H and Bonnier, Hans JRM and Bartunek, Jozef and Koolen, Jacques J},
  journal={New England Journal of Medicine},
  volume={334},
  number={26},
  pages={1703--1708},
  year={1996},
  publisher={Mass Medical Soc}
}

@article{bukavc2013fluid,
  title={Fluid-Structure Interaction in Blood Flow Capturing Non-Zero Longitudinal Structure Displacement},
  author={Buka{\v{c}}, Martina and {\v{C}}ani{\'c}, Sun{\v{c}}ica and Glowinski, Roland and Tamba{\v{c}}a, Josip and Quaini, Annalisa},
  journal={Journal of Computational Physics},
  volume={235},
  pages={515--541},
  year={2013},
  publisher={Elsevier}
}

@article{wang2018higher,
  title={A Higher-Order Discontinuous Galerkin/Arbitrary Lagrangian Eulerian Partitioned Approach To Solving Fluid--Structure Interaction Problems With Incompressible, Viscous Fluids and Elastic Structures},
  author={Wang, Yifan and Quaini, Annalisa and {\v{C}}ani{\'c}, Sun{\v{c}}ica},
  journal=jsc,
  volume={76},
  pages={481--520},
  year={2018},
  publisher={Springer}
}

@article{KingBa15,
  author={Kingma, Diederik and Ba, Jimmy},
  journal=iclr,
  title={Adam: A Method for Stochastic Optimization},
  year={2015},
  address={San Diego, CA, USA},
  volume={12},
}

@article{amari1993backpropagation,
  title={Backpropagation and Stochastic Gradient Descent Method},
  author={Amari, Shun-ichi},
  journal={Neurocomputing},
  volume={5},
  number={4-5},
  pages={185--196},
  year={1993},
  publisher={Elsevier}
}

@article{peskin1972flow,
  title={Flow Patterns Around Heart Valves: A Numerical Method},
  journal={Journal of Computational Physics},
  volume={10},
  number={2},
  pages={252-271},
  year={1972},
  issn={0021-9991},
  doi={https://doi.org/10.1016/0021-9991(72)90065-4},
  author={Charles S Peskin}
}

@article{formaggia2001coupling,
  title={On the Coupling of 3D and 1D Navier--Stokes Equations for Flow Problems in Compliant Vessels},
  author={Formaggia, Luca and Gerbeau, Jean-Fr{\'e}d{\'e}ric and Nobile, Fabio and Quarteroni, Alfio},
  journal={cmame},
  volume={191},
  number={6-7},
  pages={561--582},
  year={2001},
  publisher={Elsevier}
}

@article{zhang2024meshless,
  title={A Meshless Solver for Blood Flow Simulations in Elastic Vessels Using a Physics-Informed Neural Network},
  author={Zhang, Han and Chan, Raymond H and Tai, Xue-Cheng},
  journal={SIAM Journal on Scientific Computing},
  volume={46},
  number={4},
  pages={C479--C507},
  year={2024},
  publisher={SIAM}
}

@article{aguayo2021distributed,
  title={A Distributed Resistance Inverse Method for Flow Obstacle Identification From Internal Velocity Measurements},
  author={Aguayo, Jorge and Bertoglio, Cristobal and Osses, Axel},
  journal={Inverse Problems},
  volume={37},
  number={2},
  pages={025010},
  year={2021},
  publisher={IOP Publishing}
}

@article{kontogiannis2022joint,
  title={Joint Reconstruction and Segmentation of Noisy Velocity Images as an Inverse Navier--Stokes Problem},
  author={Kontogiannis, Alexandros and Elgersma, Scott V and Sederman, Andrew J and Juniper, Matthew P},
  journal={Journal of Fluid Mechanics},
  volume={944},
  pages={A40},
  year={2022},
  publisher={Cambridge University Press}
}

@article{kang2015nonconvex,
  title={Nonconvex Higher-Order Regularization Based Rician Noise Removal With Spatially Adaptive Parameters},
  author={Kang, Myeongmin and Kang, Myungjoo and Jung, Miyoun},
  journal={Journal of Visual Communication and Image Representation},
  volume={32},
  pages={180--193},
  year={2015},
  publisher={Elsevier}
}

@article{martin20171,
  title={On 1-Laplacian Elliptic Equations Modeling Magnetic Resonance Image Rician Denoising},
  author={Mart{\'\i}n, Adri{\'a}n and Schiavi, Emanuele and Segura de Le{\'o}n, Sergio},
  journal={Journal of Mathematical Imaging and Vision},
  volume={57},
  number={2},
  pages={202--224},
  year={2017},
  publisher={Springer}
}

@article{li2016variational,
  title={Variational Multiplicative Noise Removal by DC Programming},
  author={Li, Zhi and Lou, Yifei and Zeng, Tieyong},
  journal={Journal of Scientific Computing},
  volume={68},
  number={3},
  pages={1200--1216},
  year={2016},
  publisher={Springer}
}

@incollection{TAO1986249,
  title={Algorithms for Solving a Class of Nonconvex Optimization Problems. Methods of Subgradients},
  booktitle={Fermat Days 85: Mathematics for Optimization},
  author={Tao, Pham Dinh Tao and Souad, El Bernoussi},
  publisher={North-Holland},
  volume={129},
  pages={249-271},
  year={1986},
  issn={0304-0208},
}

@article{lou2015weighted,
  title={A Weighted Difference of Anisotropic and Isotropic Total Variation Model for Image Processing},
  author={Lou, Yifei and Zeng, Tieyong and Osher, Stanley and Xin, Jack},
  journal={SIAM Journal on Imaging Sciences},
  volume={8},
  number={3},
  pages={1798--1823},
  year={2015},
  publisher={SIAM}
}

@article{xiao2011restoration,
  title={Restoration of Images Corrupted by Mixed Gaussian-Impulse Noise via L1--L0 Minimization},
  author={Xiao, Yu and Zeng, Tieyong and Yu, Jian and Ng, Michael K},
  journal={Pattern Recognition},
  volume={44},
  number={8},
  pages={1708--1720},
  year={2011},
  publisher={Elsevier}
}

@article{chen2015convex,
  title={A Convex Variational Model for Restoring Blurred Images with Large Rician Noise},
  author={Chen, Liyuan and Zeng, Tieyong},
  journal={Journal of Mathematical Imaging and Vision},
  volume={53},
  pages={92--111},
  year={2015},
  publisher={Springer}
}

@inproceedings{getreuer2011variational,
  title={A Variational Model for the Restoration of MR Images Corrupted by Blur and Rician Noise},
  author={Getreuer, Pascal and Tong, Melissa and Vese, Luminita A},
  booktitle={International Symposium on Visual Computing},
  pages={686--698},
  year={2011},
  organization={Springer}
}

@article{goldstein2009split,
  title={The Split Bregman Method for L1-Regularized Problems},
  author={Goldstein, Tom and Osher, Stanley},
  journal={SIAM Journal on Imaging Sciences},
  volume={2},
  number={2},
  pages={323--343},
  year={2009},
  publisher={SIAM}
}

@article{zhang2024learning,
  title={A Learning-Based Framework for Topology-Preserving Segmentation Using Quasiconformal Mappings},
  author={Zhang, Han and Lui, Lok Ming},
  journal={Neurocomputing},
  volume={600},
  pages={128124},
  year={2024},
  publisher={Elsevier}
}

@article{zhang2026quasi,
  title={Quasi-conformal Convolution: A Learnable Convolution for Deep Learning on Simply Connected Open Surfaces},
  author={Zhang, Han and Ip, Tsz Lok and Lui, Lok Ming},
  journal={SIAM Journal on Imaging Sciences},
  volume={19},
  number={1},
  pages={555--583},
  year={2026},
  publisher={SIAM}
}

@article{markl20124d,
  title={4D Flow MRI},
  author={Markl, Michael and Frydrychowicz, Alex and Kozerke, Sebastian and Hope, Mike and Wieben, Oliver},
  journal={Journal of Magnetic Resonance Imaging},
  volume={36},
  number={5},
  pages={1015--1036},
  year={2012},
  publisher={Wiley Online Library}
}

@article{caselles1997geodesic,
  title={Geodesic Active Contours},
  author={Caselles, Vicent and Kimmel, Ron and Sapiro, Guillermo},
  journal={International Journal of Computer Vision},
  volume={22},
  pages={61--79},
  year={1997},
  publisher={Springer}
}

@article{lebrun2012analysis,
  title={An Analysis and Implementation of the BM3D Image Denoising Method},
  author={Lebrun, Marc},
  journal={Image Processing on Line},
  volume={2},
  pages={175--213},
  year={2012}
}

@article{zhang2018ffdnet,
  title={FFDNet: Toward a Fast and Flexible Solution for CNN-Based Image Denoising},
  author={Zhang, Kai and Zuo, Wangmeng and Zhang, Lei},
  journal={IEEE Transactions on Image Processing},
  volume={27},
  number={9},
  pages={4608--4622},
  year={2018},
  publisher={IEEE}
}

@article{yu2011dct,
  title={DCT Image Denoising: A Simple and Effective Image Denoising Algorithm},
  author={Yu, Guoshen and Sapiro, Guillermo},
  journal={Image Processing on Line},
  volume={1},
  pages={292--296},
  year={2011}
}

@book{parker2010algorithms,
  title={Algorithms for Image Processing and Computer Vision},
  author={Parker, Jim R},
  year={2010},
  publisher={John Wiley \& Sons}
}

@book{sohr2012navier,
  title={The Navier-Stokes Equations: An elementary Functional Analytic Approach},
  author={Sohr, Hermann},
  year={2012},
  publisher={Springer Science \& Business Media}
}

@article{kelliher2006navier,
  title={Navier--Stokes equations with Navier boundary conditions for a bounded domain in the plane},
  author={Kelliher, James P},
  journal={SIAM journal on mathematical analysis},
  volume={38},
  number={1},
  pages={210--232},
  year={2006},
  publisher={SIAM}
}

@article{kontogiannis2024bayesian,
  title={Bayesian inverse Navier--Stokes problems: joint flow field reconstruction and parameter learning},
  author={Kontogiannis, Alexandros and Elgersma, Scott V and Sederman, Andrew J and Juniper, Matthew P},
  journal={Inverse Problems},
  volume={41},
  number={1},
  pages={015008},
  year={2024},
  publisher={IOP Publishing}
}

@article{aguayo2022analysis,
  title={Analysis of Obstacles Immersed in Viscous Fluids Using Brinkman's Law for Steady Stokes and Navier--Stokes Equations},
  author={Aguayo, Jorge and Lincopi, Hugo Carrillo},
  journal={SIAM Journal on Applied Mathematics},
  volume={82},
  number={4},
  pages={1369--1386},
  year={2022},
  publisher={SIAM}
}

@article{aguayo2022stability,
  title={A stability result for the identification of a permeability parameter on Navier-Stokes equations},
  author={Aguayo, Jorge and Osses, Axel},
  journal={Inverse Problems},
  volume={38},
  number={7},
  pages={075001},
  year={2022},
  publisher={IOP Publishing}
}

@article{aguayo2023distributed,
  title={Distributed parameter identification for the Navier--Stokes equations for obstacle detection},
  author={Aguayo, Jorge and Bertoglio, Crist{\'o}bal and Osses, Axel},
  journal={Inverse Problems},
  volume={40},
  number={1},
  pages={015012},
  year={2023},
  publisher={IOP Publishing}
}

@article{sun2020physics,
  title={Physics-constrained bayesian neural network for fluid flow reconstruction with sparse and noisy data},
  author={Sun, Luning and Wang, Jian-Xun},
  journal={Theoretical and Applied Mechanics Letters},
  volume={10},
  number={3},
  pages={161--169},
  year={2020},
  publisher={Elsevier}
}

@article{gao2021super,
  title={Super-resolution and denoising of fluid flow using physics-informed convolutional neural networks without high-resolution labels},
  author={Gao, Han and Sun, Luning and Wang, Jian-Xun},
  journal={Physics of Fluids},
  volume={33},
  number={7},
  year={2021},
  publisher={AIP Publishing}
}

@article{wang2022dense,
  title={Dense velocity reconstruction from particle image velocimetry/particle tracking velocimetry using a physics-informed neural network},
  author={Wang, Hongping and Liu, Yi and Wang, Shizhao},
  journal={Physics of fluids},
  volume={34},
  number={1},
  year={2022},
  publisher={AIP Publishing}
}

@incollection{corona2021joint,
  title={Joint phase reconstruction and magnitude segmentation from velocity-encoded MRI data},
  author={Corona, Veronica and Benning, Martin and Gladden, Lynn F and Reci, Andi and Sederman, Andrew J and Sch{\"o}nlieb, Carola-Bibiane},
  booktitle={Time-dependent Problems in Imaging and Parameter Identification},
  pages={1--24},
  year={2021},
  publisher={Springer}
}

@book{merriman1992diffusion,
  title={Diffusion generated motion by mean curvature},
  author={Merriman, Barry and Bence, James Kenyard and Osher, Stanley},
  year={1992},
  publisher={Department of Mathematics, University of California, Los Angeles}
}

@article{zhang2025deformation,
  title={Deformation-invariant neural network and its applications in distorted image restoration and analysis},
  author={Zhang, Han and Chen, Qiguang and Lui, Lok Ming},
  journal={Neural Networks},
  pages={107378},
  year={2025},
  publisher={Elsevier}
}

@article{wang2024piratenets,
  title={Piratenets: Physics-informed deep learning with residual adaptive networks},
  author={Wang, Sifan and Li, Bowen and Chen, Yuhan and Perdikaris, Paris},
  journal={Journal of Machine Learning Research},
  volume={25},
  number={402},
  pages={1--51},
  year={2024}
}

\newpage
\appendix
\section{Proof For Theorem 3.3}
\begin{proof}\,

\textbf{Conclusion (a): $c_n \to \boldsymbol{id}$ in $L^{\infty}(D)$.} 

    \noindent Since each $f_n$ is a diffeomorphism of $D$, the limit $f_*$ is also a diffeomorphism. Thus, $f_n \to f_*$ uniformly on $D$.    
    
    \noindent By Lemma~\ref{lemma:correction}, we have $c_n \to \boldsymbol{id}$ in $L^{\infty}(D)$. This implies that the objective functional $\| \tilde{u} - u_n \circ c_n \|_{L^2(D)}$ converges to $\|\tilde{u} - u\|^2_{L^2(D)}$ as the iterations proceed given that $D$ is compact. Such convergence reflects the diminishing need for geometric correction. 

\textbf{Conclusion (b) - $f_*$ is a partial and local minimizer}

    \noindent From assumption (1) and (2), and $u_n = \mathcal{U}\left[f_n\right]$ by definition, we conclude that
    \begin{equation*}
    \mathcal{U}\left[f_n\right] \rightarrow \mathcal{U}\left[f_*\right]=u_* \quad \text { in } L^{\infty}(D),
    \end{equation*}
    which implies
    \begin{equation*}
    J_{\mathcal{U}}\left(f_n\right)=\left\|\tilde{u}-\mathcal{U}\left[f_n\right]\right\|_{L^2(D)} \rightarrow\left\|\tilde{u}-\mathcal{U}\left[f_*\right]\right\|_{L^2(D)}=J_{\mathcal{U}}\left(f_*\right).
    \end{equation*}

\textbf{- Step 1:}

    \noindent Let $\varepsilon>0$ be arbitrary. By assumption (1), $u_n \rightarrow u_*$ in $L^{\infty}(D)$. So there exists $N_1$ such that for all $n \geq N_1$,
    \begin{equation*}
    \left\|u_n-u_*\right\|_{L^{\infty}(D)}<\frac{\varepsilon}{3}.
    \end{equation*}
    Now fix a small perturbation $c \in \mathcal{H}(D)$, close to identity, which we will quantify later. Define
    \begin{equation*}
    f=c^{-1} \circ f_n \quad\text{and}\quad \mathcal{U}[f]:=\mathcal{U}\left[c^{-1} \circ f_n\right].
    \end{equation*}
    With $c_n=\arg \min _{c \in \mathcal{H}(D)}\left\|\tilde{u}-u_n \circ c\right\|_{L^2(D)}$, we have
    \begin{equation}
    \left\|\tilde{u}-u_n \circ c_n\right\|_{L^2(D)} \leq\left\|\tilde{u}-u_n \circ c\right\|_{L^2(D)}, \quad \forall c \in \mathcal{H}(D).
    \label{eq:inequal_cn}
    \end{equation}

\textbf{- Step 2: Inequality for $\left\|\mathcal{U}\left[c^{-1} \circ f_n\right]-\mathcal{U}\left[f_n\right]\right\|_{L^2(D)}$}

    \noindent By the continuity of the mapping in assumption (3), for every $\varepsilon>0$, there exists $\delta>0$ such that for $f, f_n \in \mathcal{H}(D)$,
    \begin{equation*}
    \left\|f-f_n\right\|_{L^{\infty}(D)}<\delta \quad \Rightarrow \quad\left\|\mathcal{U}\left[f\right]-\mathcal{U}\left[f_n\right]\right\|_{L^2(D)}<\frac{\varepsilon}{3} .
    \end{equation*}
    By conclusion (a), we have $\|c^{-1} \circ f_n-f_n \|_{L^\infty(D)}< \delta$ for $n$ large enough. Hence, we get
    \begin{equation}
    \left\|\mathcal{U}\left[c^{-1} \circ f_n\right]-\mathcal{U}\left[f_n\right]\right\|_{L^2(D)} < \frac{\varepsilon}{3}.
    \label{eq:step2}
    \end{equation}

\textbf{- Step 3: Inequality for $\left\|u_n \circ c-u_n\right\|_{L^2(D)}$}

    \noindent As $u_n$ satisfies \eqref{eq:ns_reform} for fixed $f_n$, with assumption (4), there is a constant $M$ such that
    \begin{equation*}
    |u_n(c(x))-u_n(x)|^2\leq M^2|c(x)-x|^2.
    \end{equation*}
    
    \noindent Integrating over $D$, we get
    \begin{equation*}
    \left\|u_n \circ c-u_n\right\|_{L^2(D)}^2=\int_D\left|u_n(c(x))-u_n(x)\right|^2 d x \leq M^2 \int_D|c(x)-x|^2 d x
    \end{equation*}
    Now since $\|c-\mathrm{id}\|_{L^{\infty}(D)}<\delta$, if we choose
    \begin{equation*}
    \delta<\frac{\varepsilon}{3 M \sqrt{|D|}}
    \end{equation*}
    We can have
    \begin{equation}
    \left\|u_n \circ c-u_n\right\|_{L^2(D)} < M \delta \sqrt{|D|} < \frac{\varepsilon}{3}    
    \label{eq:step3}
    \end{equation}

\textbf{- Step 4:}

    \noindent With \eqref{eq:step2} and \eqref{eq:step3}, we have
    \begin{equation}
    \left\|u_n \circ c-\mathcal{U}\left[c^{-1} \circ f_n\right]\right\|_{L^2(D)} 
    \leq \left\|u_n \circ c-u_n\right\|_{L^2(D)}+\left\|u_n-\mathcal{U}\left[c^{-1} \circ f_n\right]\right\|_{L^2(D)} < \frac{2 \varepsilon}{3}.
    \label{eq:error_composite_solution}
    \end{equation}
    
    \noindent We now compare
    \begin{equation*}
    J_{\mathcal{U}}(f)=\left\|\tilde{u}-\mathcal{U}\left[c^{-1} \circ f_n\right]\right\|_{L^2(D)}, \quad J_{\mathcal{U}}\left(f_{n+1}\right)=\left\|\tilde{u}-\mathcal{U}\left[c_n^{-1} \circ f_n\right]\right\|_{L^2(D)}.
    \end{equation*}    
    By the triangle inequality
    \begin{equation*}
    \left\|\tilde{u}-\mathcal{U}\left[c^{-1} \circ f_n\right]\right\| \geq\left\|\tilde{u}-u_n \circ c\right\|-\left\|u_n \circ c-\mathcal{U}\left[c^{-1} \circ f_n\right]\right\| .
    \end{equation*}    
    Using \eqref{eq:error_composite_solution}
    \begin{equation*}
    \left\|\tilde{u}-\mathcal{U}\left[c^{-1} \circ f_n\right]\right\| > \left\|\tilde{u}-u_n \circ c\right\|-\frac{2 \varepsilon}{3}.
    \end{equation*}    
    Together with \eqref{eq:inequal_cn}, we have
    \begin{equation}
    J_{\mathcal{U}}(f) \geq\left\|\tilde{u}-u_n \circ c_n\right\|-\frac{2 \varepsilon}{3}.
    \end{equation}
    Similarly, we can have $\left\|u_n \circ c_n-\mathcal{U}\left[c_n^{-1} \circ f_n\right]\right\|<\varepsilon / 3$ by the techniques as above, and obtain
    \begin{equation}
    \left\|\tilde{u}-u_n \circ c_n\right\| \geq \left\|\tilde{u}- \mathcal{U}\left[f_{n+1}\right]\right\| - \left\|\mathcal{U}\left[f_{n+1}\right] - u_n \circ c_n\right\| > J\left(f_{n+1}\right)-\frac{\varepsilon}{3}.
    \end{equation}    
    Hence
    \begin{equation*}
    J_{\mathcal{U}}(f) > J_{\mathcal{U}}\left(f_{n+1}\right)-\varepsilon.    
    \end{equation*}    
    
    \noindent Take $n \rightarrow \infty$, then $f_{n+1} \rightarrow f_*, \mathcal{U}\left[f_{n+1}\right] \rightarrow \mathcal{U}\left[f_*\right]$, so
    \begin{equation}
    \liminf _{f \rightarrow f_*} J_{\mathcal{U}}(f) \geq J_{\mathcal{U}}\left(f_*\right).
    \end{equation}    
    Thus, $f_*$ is a partial and local minimizer of the misfit.

\textbf{Conclusion (c) - $u_*$  is a partial and local minimizer}

    \noindent Denote the solution set for $u$ as
    \begin{equation*}
    \mathcal{A}:=\left\{\mathcal{U}\left[f\right] \in L^2(D) \mid f \in \mathcal{H}(D)\right\} .
    \end{equation*} 
    Define the functional as
    \begin{equation*}
    J(u) = \|\tilde{u}-u\|_{L^2(D)}.
    \end{equation*}     
    As $f_* \in \mathcal{H}(D)$ is a local minimizer for $f \mapsto J(\mathcal{U}\left[f\right])$, there exists $\delta_1>0$ such that for all $f \in \mathcal{H}(D)$ and $\left\|f-f_*\right\|_{L^{\infty}(D)}<\delta_1$ such that
    \begin{equation}
    J(\mathcal{U}\left[f\right])=\|\tilde{u}-\mathcal{U}\left[f\right]\|_{L^2(D)} \geq\left\|\tilde{u}-\mathcal{U}\left[f_*\right]\right\|_{L^2(D)}=J\left(u_*\right).
    \label{eq:local_f}
    \end{equation} 
    
    \noindent Let $\varepsilon>0$, by assumption (3), there exists $\delta_2>0$ such that for all $f \in \mathcal{H}(D)$, which satisfy $\left\|f-f_*\right\|_{L^{\infty}(D)} < \delta_2$, we have
    $$
    \left\|\mathcal{U}\left[f\right]-\mathcal{U}\left[f_*\right]\right\|_{L^2(D)}<\varepsilon,
    $$
    where we can make $\delta_2<\delta_1$ without loss of generality.
    
    \noindent Then, for any $u\in \mathcal{A}$ and $\left\|\mathcal{U}\left[f\right]-\mathcal{U}\left[f_*\right]\right\|_{L^2(D)}<\varepsilon$, we have
    \begin{equation*}
     \left\|f-f_*\right\|_{L^{\infty}(D)}<\delta_2<\delta_1.
    \end{equation*}     
    By \eqref{eq:local_f}, we have 
    \begin{equation}
    J(u)=J(\mathcal{U}\left[f\right]) \geq J\left(\mathcal{U}\left[f_*\right]\right)=J\left(u_*\right),
    \end{equation} 
    which means $u_*$ is a partial and local minimizer of $J$ over $\mathcal{A}$.
\end{proof}

In the theory below, we show that if the sequence of mappings $f_n$, obtained from some iterative algorithm for minimizing $\|\tilde{u} - u\|^2_{L^2(D)}$ under \eqref{eq:ns_reform}, converges, then the corresponding correction mappings $c_n = f_{n+1}\circ f_{n}^{-1}$ will converge to the identity mapping.

\begin{lemma}
    Let $D \subset \mathbb{C}$ be a compact set. Given a sequence of diffeomorphisms $f_n : D \to D$. We have    
    \begin{equation}
        f_n \to f \quad \text{uniformly on } D, \text{ where $f: D \to D$ is a diffeomorphism,}
    \end{equation}
    if and only if
    \begin{equation}
        f_{n+1} \circ f_n^{-1} \to \boldsymbol{id}_D \quad \text{uniformly on } D.
    \end{equation}
    \label{lemma:correction}
\end{lemma}

\begin{proof}\,

    \noindent ($=>$):
    
    \noindent Let $\varepsilon > 0$ be arbitrary. Since $f_n \to f$ uniformly on $D$, there exists $N_1 \in \mathbb{N}$ such that for all $n \geq N_1$,
    \begin{equation}
    \sup_{x \in D} \| f_n(x) - f(x) \| < \frac{\varepsilon}{2}.
    \end{equation}
    Additionally, the uniform convergence of $f_n \to f$ and the fact that each $f_n$ and $f$ are diffeomorphisms implies that $f_n^{-1} \to f^{-1}$ uniformly on $D$ by Lemma~\ref{lemma:inverse_function}.
    
    Let $N = \max(N_1, N_2)$. For all $n \geq N$ and all $x \in D$, define $y := f_n^{-1}(x)$. Then $x = f_n(y)$, and we compute
    \begin{equation}
    f_{n+1} \circ f_n^{-1}(x) = f_{n+1}(y),
    \end{equation}
    so that
    \begin{equation}
    \| f_{n+1} \circ f_n^{-1}(x) - x \| = \| f_{n+1}(y) - f_n(y) \|.
    \end{equation}
    Now, we insert $f(y)$ and use the triangle inequality
    \begin{equation}
    \| f_{n+1}(y) - f_n(y) \| \leq \| f_{n+1}(y) - f(y) \| + \| f(y) - f_n(y) \| < \frac{\varepsilon}{2} + \frac{\varepsilon}{2} = \varepsilon.
    \end{equation}    
    Since this holds for all $x \in D$, we have
    \begin{equation}
    \sup_{x \in D} \| f_{n+1} \circ f_n^{-1}(x) - x \| < \varepsilon, \quad\forall n \geq N.
    \end{equation}
    Hence,
    \begin{equation}
    f_{n+1} \circ f_n^{-1} \to \boldsymbol{id}_D \quad \text{uniformly on } D.
    \end{equation}
    
    \noindent ($<=$):
    
    \noindent With $f_{n+1} \circ f_n^{-1} \to \boldsymbol{id}_D$ uniformly on $D$ for every $\varepsilon > 0$, there exists $N \in \mathbb{N}$ such that for all $n \geq N$ and all $x \in D$,
    \begin{equation}
        |f_{n+1}(f_n^{-1}(x)) - x| < \varepsilon.
    \end{equation}
    Let $y = f_n^{-1}(x)$, so $x = f_n(y)$. Then the above becomes
    \begin{equation}
        |f_{n+1}(y) - f_n(y)| < \varepsilon, \quad \forall y \in D.
    \end{equation}
    Thus,
    \begin{equation}
        \sup_{y \in D} |f_{n+1}(y) - f_n(y)| < \varepsilon.
    \end{equation}
    This shows that the sequence $\{f_n\}$ is uniformly Cauchy in $C(D)$, the space of continuous functions on the compact set $D$ with the sup norm.
    
    \noindent Since $D$ is compact and each $f_n$ is continuous, there exists a continuous function $f : D \to D$ such that
    \begin{equation}
        f_n \to f \quad \text{uniformly on } D.
    \end{equation}
\end{proof}

\begin{lemma} 
    Let $D \subset \mathbb{C}$ be a compact set, and let $\{f_n\}_{n=1}^\infty$ be a sequence of diffeomorphism from $D$ onto itself such that $f_n \to f$ uniformly on $D$. Assume that the limit function $f: D \to D$ is also a diffeomorphism. Then
    \begin{equation}
    f_n^{-1} \to f^{-1} \quad \text{uniformly on } D.
    \end{equation}
    \label{lemma:inverse_function}
\end{lemma}
\begin{proof}    
    For any $\varepsilon > 0$. Since $f_n^{-1}$ is continuous on the compact set $D$, there exists $\delta > 0$ such that for any $x_1, x_2 \in D$,
    \begin{equation}
        |x_1 - x_2| < \delta \quad \Rightarrow \quad |f_n^{-1}(x_1) - f_n^{-1}(x_2)| < \varepsilon.
        \label{neq:finv_continuous}
    \end{equation}
    
    \noindent Now fix any $x \in D$, and let $y = f^{-1}(x)$. We can have $x = f(y)$ and $x_n := f_n(y)$. For $\delta$ mentioned above, we can find a sufficiently large $n$, as $f_n \to f$ uniformly, to obtain
    \begin{equation}
        |x_n - x| = |f_n(y) - f(y)| < \delta.
    \end{equation}
    
    \noindent Since $f_n$ is bijective, we have $f_n^{-1}(x_n) = y = f^{-1}(x)$. Together with \eqref{neq:finv_continuous}, we have
    \begin{equation}
    \begin{aligned}
        |f_n^{-1}(x) - f^{-1}(x)| &\leq |f_n^{-1}(x) - f_n^{-1}(x_n)| + |f_n^{-1}(x_n) - f^{-1}(x)|  \\ 
        &< \varepsilon + 0 = \varepsilon.
    \end{aligned}
    \end{equation}
    
    \noindent Therefore, $f_n^{-1} \to f^{-1}$ uniformly on $D$.
\end{proof}

\section{Metrics Definition}\label{app:metric}

To evaluate the accuracy and speed of our methods, some accuracy measurements are reported in the following sections. 

The relative error (RE) is defined as
\begin{equation}
    \mathrm{RE}(\boldsymbol{u}^*,\boldsymbol{u}^{\text{GT}}) = 
    \sum_{\boldsymbol{x}\in\mathcal{G}} \frac{|\boldsymbol{u}^*(\boldsymbol{x}) - \boldsymbol{u}^{\text{GT}}(\boldsymbol{x})|^2}{|\boldsymbol{u}^{\text{GT}}(\boldsymbol{x})|^2},
\end{equation}
where $\boldsymbol{u}^*$ is the reconstructed flow, $\boldsymbol{u}^{\text{GT}}$ is the ground truth flow, and $\mathcal{G}$ denotes the set of pixel points.

Peak Signal-to-Noise Ratio (PSNR) is defined as
\begin{equation}
    \mathrm{PSNR}(\boldsymbol{u}^*,\boldsymbol{u}^{\text{GT}}) = 10 \log_{10}\left(\frac{\max_{\boldsymbol{x}\in\mathcal{G}}(\boldsymbol{u}^{\text{GT}}(\boldsymbol{x}))^2}{\frac{1}{|\mathcal{G}|}\sum_{\boldsymbol{x}\in\mathcal{G}}|\boldsymbol{u}^*(\boldsymbol{x}) - \boldsymbol{u}^{\text{GT}}(\boldsymbol{x})|^2}\right),
\end{equation}
where $|\mathcal{G}|$ is number of points in $\mathcal{G}$.

Structural Similarity Index (SSIM) is defined as
\begin{equation}
    \mathrm{SSIM}(\boldsymbol{u}^*,\boldsymbol{u}^{\text{GT}}) = 
    \frac{1}{|\mathcal{G}|}\sum_{\boldsymbol{x}\in\mathcal{G}}\frac{\left(2\mu_{\boldsymbol{u}^*}(\boldsymbol{x}) \mu_{\boldsymbol{u}^{\mathrm{GT}}}(\boldsymbol{x}) + C_1\right)
    \left(2\sigma_{\boldsymbol{u}^*,\boldsymbol{u}^{\mathrm{GT}}}(\boldsymbol{x}) + C_2\right)}
    {\left(\mu_{\boldsymbol{u}^*}(\boldsymbol{x})^2 + \mu_{\boldsymbol{u}^{\mathrm{GT}}}(\boldsymbol{x})^2 + C_1\right)
    \left(\sigma_{\boldsymbol{u}^*}(\boldsymbol{x})^2 + \sigma_{\boldsymbol{u}^{\mathrm{GT}}}(\boldsymbol{x})^2 + C_2\right)},
\end{equation}
where $C_1, C_2$ are small constants to stabilize the division, and
\begin{align*}
&\mu_{\boldsymbol{u}^*} = \boldsymbol{u}^* \ast w, 
\quad\quad\,
\sigma^2_{\boldsymbol{u}^*} = (\boldsymbol{u}^*)^2 \ast w - (\mu_{\boldsymbol{u}^*})^2,\\
&\mu_{\boldsymbol{u}^{\text{GT}}} = \boldsymbol{u}^{\text{GT}} \ast w, \quad
\sigma^2_{\boldsymbol{u}^{\text{GT}}} = (\boldsymbol{u}^{\text{GT}})^2 \ast w - (\mu_{\boldsymbol{u}^{\text{GT}}})^2,
\\
&\sigma_{\boldsymbol{u}^*, \boldsymbol{u}^{\text{GT}}} = (\boldsymbol{u}^* \cdot \boldsymbol{u}^{\text{GT}}) \ast w - \mu_{\boldsymbol{u}^*} \cdot \mu_{\boldsymbol{u}^{\text{GT}}},
\end{align*}
where $\omega$ is a Gaussian kernel of mean as $0$ and standard deviation as $1.5$. The convolution is with reflect padding

The Dice coefficient measures the spatial overlap of two binary masks and is defined as
\begin{equation}
    \mathrm{Dice}(\Omega_*, \Omega_{\mathrm{GT}}) = \frac{2|\Omega_* \cap \Omega_{\mathrm{GT}}|}{|\Omega_*| + |\Omega_{\mathrm{GT}}|},
\end{equation}
where $\Omega_*$ and $\Omega_{\mathrm{GT}}$ represent the predicted and ground-truth regions. $|\Omega|$ is the area of $\Omega$. A value of $1$ indicates perfect overlap.

The 95th percentile Hausdorff distance (HD95) measures the boundary discrepancy between two segmented regions by taking the 95th percentile of the bidirectional surface distances, thereby reducing the influence of outlier points or noisy boundaries. It is defined as
\begin{equation}
\mathrm{HD}_{95}(\Omega_*, \Omega_{\mathrm{GT}}) = 
\max \Big\{
\mathrm{P}_{95}\big( \{ d(x, \Omega_{\mathrm{GT}}) \mid x \in \Omega_* \} \big), \;
\mathrm{P}_{95}\big( \{ d(y, \Omega_*) \mid y \in \Omega_{\mathrm{GT}} \} \big)
\Big\},
\end{equation}
where $d(x, \Omega_{\mathrm{GT}}) = \inf_{y \in \Omega_{\mathrm{GT}}} \|x - y\|$, $d(y, \Omega_*) = \inf_{x \in \Omega_*} \|y - x\|$ and $\mathrm{P}_{95}(\{z_i\})$ denotes the 95\% smallest of the set of $\{z_i\}$.

\end{document}